\documentclass[3p,12pt]{mystyle}

\usepackage[T1]{fontenc}
\usepackage[latin1]{inputenc}

\usepackage{amsmath,amssymb,theorem}

\usepackage{amsfonts}
\usepackage{wasysym}

\usepackage{epsfig}
\usepackage{epic}
\usepackage{eepic}
\usepackage{graphics}
\usepackage{graphicx}
\usepackage{color}
\usepackage{longtable}
\usepackage{subfigure}

\newcommand{\ds}{\displaystyle}
\newcommand{\ts}{\textstyle}

\newcommand{\Ord}{\mathcal{O}}

\renewcommand{\Re}{\operatorname{Re}}
\renewcommand{\Im}{\operatorname{Im}}

\newcommand{\Nn}{{\mathbb N}}

\newcommand{\Rr}{{\mathbb R}}

\newcommand{\Zz}{{\mathbb Z}}

\newcommand{\vph}{\varphi}

\newcommand{\Sd}{\mathbb{S}^2}

\newcommand{\vect}[1]{\boldsymbol{#1}}

\DeclareMathOperator{\Iop}{I}
\newcommand{\I}{\boldsymbol{\Iop}}
\newcommand{\Imm}{\I^{(\vect{m})}}
\newcommand{\Immp}{\I^{(\vect{m})}}
\newcommand{\Imms}{\I_{\mathrm{S}}^{(\vect{m})}}

\DeclareMathOperator{\Jop}{J}
\newcommand{\J}{\vect{\Jop}}
\newcommand{\Jmm}{\J^{(\vect{m})}}

\DeclareMathOperator{\LSop}{LS}
\newcommand{\LS}{\boldsymbol{\LSop}}

\newcommand{\LSm}{\LS^{(\vect{m})}}

\newcommand{\Gam}{\vect{\Gamma}^{(\vect{m})}}
\newcommand{\Gamstar}{\vect{\Gamma}^{(\vect{m}),*}}
\newcommand{\Gamsup}{\overline{\vect{\Gamma}}^{(\vect{m})}}

\newcommand{\Gamup}{\vect{\Gamma}^{(\vect{m}),\mathrm{U}}}
\newcommand{\Gamdown}{\vect{\Gamma}^{(\vect{m}),\mathrm{D}}}

\newcommand{\Lissajous}{\vect{\ell}^{(\vect{m})}_{\alpha}}

\newcommand{\Pim}{\Pi^{(\vect{m})}}
\newcommand{\Pimreal}{\Pi^{(\vect{m})}_{\mathcal{R}}}
\newcommand{\Pims}{\Pi^{(\vect{m})}_{\mathrm{S}}}

\newcommand{\Dx}[1]{\mathrm{d}#1}

\newcommand{\polbas}{\chi^{(\vect{m})}_{\vect{\gamma}}}
\newcommand{\polbasreal}{\chi^{(\vect{m})}_{\mathcal{R},\vect{\gamma}}}
\newcommand{\funpol}{\mathcal{L}}
\newcommand{\funpols}{\mathcal{L}_{\mathrm{S}}}

\newcommand{\polbascont}{{X}_{\vect{\gamma}}}
\newcommand{\polbascontreal}{{X}_{\mathcal{R},\vect{\gamma}}}

\usepackage[scr=boondoxo,scrscaled=1.05]{mathalfa}
\newcommand{\ff}{\mathscr{f}}

\newtheorem{theorem}{Theorem}

\newtheorem{proposition}[theorem]{Proposition}
\newtheorem{corollary}[theorem]{Corollary}
\newdefinition{remark}{Remark}
\newdefinition{example}{Example}
\newproof{proof}{Proof}

\begin{document}

\begin{frontmatter}
\title{A spectral interpolation scheme on the unit sphere based on the nodes of spherical Lissajous curves}

\author{Wolfgang Erb}
\ead{erb@math.hawaii.edu}

\address{ University of Hawai'i at M\=anoa, 
          Department of Mathematics \\ 2565 McCarthy Mall, Keller Hall 401A,
          Honolulu, HI, 96822
          }

\date{\today}

\begin{abstract}
For sampling values along spherical Lissajous curves we establish a spectral interpolation and quadrature scheme on the sphere.
We provide a mathematical analysis of spherical Lissajous curves and study the characteristic properties of their intersection
points. Based on a discrete orthogonality structure we are able to prove the unisolvence of the interpolation problem.
As basis functions for the interpolation space we use a parity-modified double Fourier basis on the sphere which allows us to implement the interpolation scheme
in an efficient way. We further show that the numerical condition number of the interpolation scheme displays a logarithmic growth. 
 As an application, we use the developed interpolation algorithm to estimate the rotation of an object based on measurements at
the spherical Lissajous nodes.
\end{abstract}


\end{frontmatter}

\section{Introduction} \label{sec:introduction}

In Magnetic Resonance Imaging, motions of the scanned subject during the imaging process cause artifacts in the reconstruction. One concept to
detect and correct subject motions in 3D is based on the additional measurements of spherical navigator echoes \cite{Costa2010,Welchetal2002}. These measurements are
performed along sampling trajectories on a spherical shell and used to estimate rotations and translations of the scanned subject.
Particular promising trajectories for such navigator measurements are spherical Lissajous curves \cite{Ullisch2012}.

A \emph{spherical Lissajous curve} is given in parametric form as
\begin{equation} \label{201509161237}
\Lissajous(t) = \Big( \sin(m_2 t) \cos (m_1 t - \alpha \pi), \ \sin(m_2 t) \sin (m_1 t - \ts \alpha \pi), \ \cos(m_2 t)  \Big), \quad t \in \mathbb{R},
\end{equation}
with a frequency vector $\vect{m} = (m_1, m_2) \in\Nn^{2}$ and a rotation parameter $\alpha \in \Rr$. The curve $\Lissajous$
lies in the unit sphere $\Sd = \{\vect{x} \in \Rr^3:\; x_1^2+x_2^2+x_3^2 = 1\}$ of the three-dimensional space $\Rr^3$. Similar as for bivariate Lissajous curves \cite{DenckerErb2017a,DenckerErb2015a,Erb2015,ErbKaethnerAhlborgBuzug2015},
the curve $\vect{\ell}^{(\vect{m})}_{\alpha}$
describes a superposition of a latitudinal and a longitudinal harmonic motion determined by the frequencies $m_1$ and $m_2$.

The goal of this article is to provide a mathematical analysis of spherical Lissajous curves and to study their role as generating curves
for spherical interpolation. Of particular interest for our analysis are the intersection points $\LSm$ of one or more spherical Lissajous curves. These 
Lissajous nodes provide a good measure of how densely the curves cover the sphere. Further, these nodes are relevant for applications. In \cite{Ullisch2012}, the intersection points of spherical Lissajous curves are used to correct
spin-spin relaxation effects for the navigator measurements. 

In the upcoming sections, we will give two different characterizations of the Lissajous nodes $\LSm$. Particularly interesting for us and for applications is the case
when the two frequencies $m_1$ and $m_2$ are relatively prime and $m_2$ is even. In this case, the nodes $\LSm$ can be described as time equidistant samples along a single spherical Lissajous curve. The restriction to even numbers $m_2$ guarantees that the two poles of $\Sd$ are included in $\LSm$. The second characterization in Section \ref{sec:nodesphere} is based on particularly defined index sets $\Immp$ and allows an explicit description of the nodes $\LSm$. It includes also the case when $m_1$ and $m_2$ are not relatively prime and when more than one Lissajous curve is needed to generate the nodes $\LSm$.

The main task of this manuscript is to use the spherical Lissajous nodes $\LSm$ to derive a novel scheme for interpolation and quadrature on the sphere.
To this end, we transfer concepts and techniques developed in \cite{DenckerErb2017a,DenckerErb2015a} for multivariate polynomial interpolation 
on Lissajous-Chebyshev nodes in the hypercube $[-1,1]^{\mathsf{d}}$ to a corresponding setting in spherical coordinates. 
As for the multivariate Lissajous-Chebyshev points, a main step in the proof of the spherical interpolation scheme is a
discrete orthogonality structure linked to the spherical Lissajous nodes. This discrete orthogonality structure will be derived in Section \ref{1507091240}.

As a basis system for the interpolation on the Lissajous nodes we will use a parity-modified double Fourier basis in spherical coordinates. These basis functions
were introduced in the 70's \cite{Merilees1973,Orszag1974} as a stable alternative to
the spherical harmonics and the Robert functions. Since then, they were used in a series of applications on the sphere
as for instance described in \cite[Section 18.27]{Boyd2000} and \cite{BoerSteinberg1975,Boyd1978,Fornberg1995,Ganesh1998,TownsendWilberWright2016}. As these functions are directly built on a Fourier series they are very well suited for computational purposes. Compared to spherical harmonics there are however some issues at the poles of the
sphere which have to be treated properly. For a more detailed discussion on different aspects of this basis system we refer to the treatises given in \cite{Boyd1978,Boyd2000}.

For the actual work the concrete form of the parity-modified Fourier basis plays a crucial role. We will establish a close link between interpolation
on the nodes of the spherical Lissajous curves and this basis system. This discussion will lead to Theorem \ref{201512131945} in which we prove the
uniqueness of the interpolation in spaces spanned by the double Fourier basis. In Section \ref{sec:implementation}, we will discover that the
mentioned structure of the basis functions leads to an efficient implementation of the interpolation scheme in terms of a double Fourier transform. In Section \ref{sec:convergence} we will further give a short description of the numerical condition number and the convergence of the interpolation scheme. We will see that, similar to spectral methods on the hypercube $[-1,1]^{\mathsf{d}}$, the numerical condition displays a slow logarithmic growth and that the interpolant converges fast if the data values are derived from smooth functions. 

Finally, we will provide some numerical experiments and present an idea on how this novel interpolation scheme on the sphere can be applied
to estimate rotations of a three-dimensional object based on measurements along spherical Lissajous curves.

\section{Spherical Lissajous curves} \label{sec:lisasphere}

\begin{figure}[htb]
	\centering
	\subfigure[\hspace*{1em} The curve $\vect{\ell}^{(7,6)}_{0}$ and the nodes $\LS^{(7,6)}_{0}$.
	]{\includegraphics[scale=0.9]{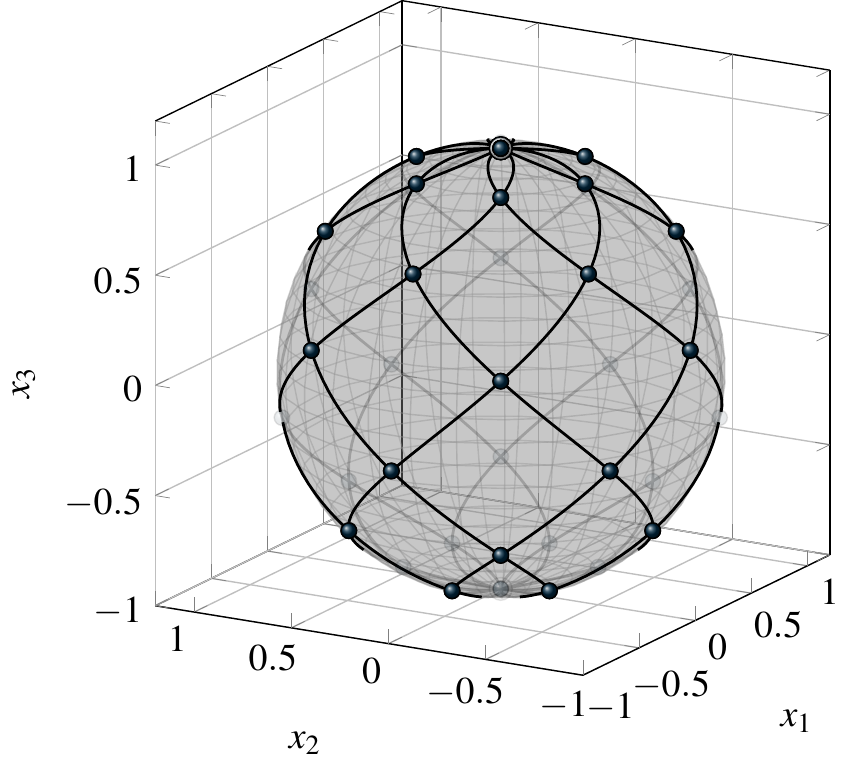}}
	\hfill	
	\subfigure[\hspace*{1em} The curve $\vect{\ell}^{(6,5)}_{0}$ and its intersection points.
	]{\includegraphics[scale=0.9]{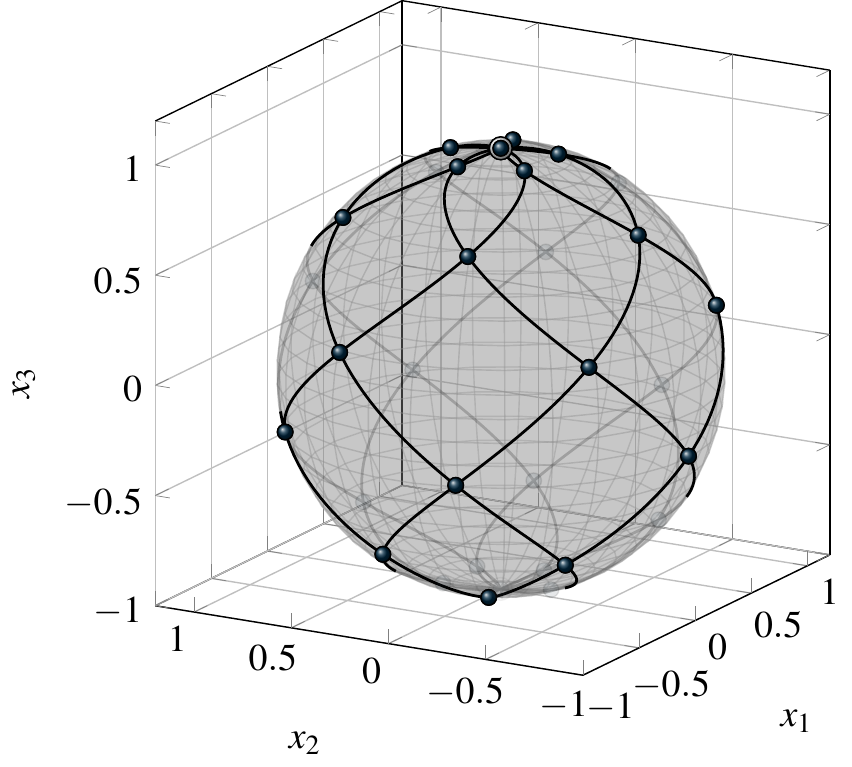}}
  	\caption{Illustration of two Lissajous curves and its intersection points described in Proposition \ref{prop-11}. The intersection
  	points of the curve are marked with black dots. The circled dots describe the two poles of $\Sd$.
  	} \label{fig:LS-1}
\end{figure}

In a first step, we want to derive some fundamental properties of the spherical Lissajous curve $\Lissajous$ defined in \eqref{201509161237}. In particular,
we are interested in its minimum period and in the number and type of its self-intersection points. Two examples of spherical
Lissajous curves with their intersection points are illustrated in Figure \ref{fig:LS-1}.

If the frequencies $m_1$ and $m_2$ of $\Lissajous$ are relatively prime, Proposition \ref{prop-11} below implies that the minimum period
of $\Lissajous$ is $2\pi$. In general, if $g = \mathrm{gcd}(\vect{m})$ denotes the greatest common divisor of $m_1$ and $m_2$,
then $\Lissajous$ can be rewritten as $\Lissajous(t) = \vect{\ell}^{(\vect{m}/g)}_{\alpha}( g t)$ and
the minimum period of $\Lissajous$ is $2 \pi/g$. For the description of spherical
Lissajous curves it is therefore enough to restrict ourselves for the moment to tuples $\vect{m}$ of relatively prime numbers. 

To extract the self-intersection points of the curve $\Lissajous$ we consider for $t \in [0,2\pi)$
the sets $\mathcal{A}^{(\vect{m})}(t) = \{ s \in [0,2\pi): \ \Lissajous(s) = \Lissajous(t) \}$ and the sampling points
\begin{align} t^{(\vect{m})}_{l} &= \frac{l \pi}{m_1 m_2}, \quad l \in \{0,1, \ldots, 2m_1m_2-1\}, \label{eq-samples1}\\
 t^{(\vect{m})}_{l+\frac12} &= \frac{(l + \frac12) \pi}{m_1 m_2}, \quad l \in \{0,1, \ldots, 2m_1m_2-1\}. \label{eq-samples2} \
\end{align}

\begin{proposition} \label{prop-11}Let $\mathrm{gcd}(\vect{m}) = 1$. If $m_2$ is even, then
 \[  \begin{array}{lll}
      (i) & \# \mathcal{A}^{(\vect{m})}(t) = m_2  & \text{if}\quad t \in \{\,t^{(\vect{m})}_{l}\,| \ l\in \{0,\ldots,2m_1m_2-1\}, \ l \equiv 0 \mod m_1 \},\\
      (ii) & \# \mathcal{A}^{(\vect{m})}(t) = 2 & \text{if}\quad t \in \{\,t^{(\vect{m})}_{l}\,| \ l\in \{0,\ldots,2m_1m_2-1\}, \ l \not\equiv 0 \mod m_1 \}, \\
      (iii) & \# \mathcal{A}^{(\vect{m})}(t) = 1 & \text{if}\quad t \in [0,2\pi) \setminus \{\,t^{(\vect{m})}_{l}\,|\, l\in \{0,\ldots,2m_1m_2-1\}\,\}.
       \end{array}
\]
If $m_2$ is odd, then
 \[  \begin{array}{lll}
      (i)' & \# \mathcal{A}^{(\vect{m})}(t) = m_2  & \text{if}\quad t \in \{\,t^{(\vect{m})}_{l}\,| \ l\in \{0,\ldots,2m_1m_2-1\}, \ l \equiv 0 \mod m_1 \}, \\
      (ii)' & \# \mathcal{A}^{(\vect{m})}(t) = 2 & \text{if}\quad t \in \{\,t^{(\vect{m})}_{l+\frac12}\,| \ l\in \{0,\ldots,2m_1m_2-1\} \}, \\
      (iii)' & \# \mathcal{A}^{(\vect{m})}(t) = 1 & \text{for all other $t \in [0,2\pi)$.}
      \end{array}
\]
\end{proposition}

\begin{remark}
If $m_1$ and $m_2$ are relatively prime we obtain as a consequence of Proposition \ref{prop-11} that the minimum period
of $\Lissajous$ is $2 \pi$. The points $\Lissajous(t)$ with $\# \mathcal{A}^{(\vect{m})}(t) = m_2$ correspond to the north or the south
pole of the sphere. Hence, in every period $[0,2\pi)$ the curve $\Lissajous$ traverses both poles $m_2$ times. All other points
$\Lissajous(t)$ with $\# \mathcal{A}^{(\vect{m})}(t) = 2$ correspond to non-polar double points of the curve on the sphere, i.e., they
are traversed twice by the curve as $t$ varies from $0$ to $2 \pi$. Depending on whether $m_1$ is odd or even, we get a different number of 
self-intersection points for $\Lissajous$. These numbers are summarized in Table \ref{tab:11}. 
\end{remark}

\begin{table}[htb] 
 \caption{Number and type of intersection points (IP's) for the curve $\Lissajous$ if $m_1$, $m_2$ are 
 relatively prime. For general $\vect{m}$, the corresponding numbers are obtained 
 by considering the curve $\vect{\ell}^{(\vect{m}/g)}_{\alpha}$ instead ($g = \mathrm{gcd}(\vect{m})$). } \label{tab:11} 
 \begin{center}
 \begin{tabular}[t]{lll} \hline \noalign{\smallskip}
  Curve $\Lissajous$         & Number of IP's& Type of IP's \\ \noalign{\smallskip}\hline \noalign{\smallskip}
  $m_2$ even   & $m_2 (m_1 -1) + 2$ & $\begin{array}{l} \text{2 poles, traversed $m_2$ times in one period,} \\ 
                                                      \text{$m_2 (m_1 -1)$ non-polar double points} \end{array}$ \\ \noalign{\medskip}
  $m_2>1$ odd  & $ m_1 m_2 + 2$ & $\begin{array}{l} \text{2 poles, traversed $m_2$ times in one period,} \\ 
                                                       \text{$m_1 m_2$ non-polar double points} \end{array}$ \\ \noalign{\medskip}
  $m_2=1$      & $ m_1 $ & $\begin{array}{l}           \text{$m_1$ non-polar double points} \end{array}$  \\ \hline 
  \end{tabular}    
  \end{center}
\end{table}

\begin{proof} We use the equivalence relation $s\eqsim t$ to denote that $t-s\in 2\pi\mathbb{Z}$. We consider first all $t \in [0,2\pi)$ such
that $\Lissajous(t)$ is one of the poles $(0,0,v)$, $v\in \{-1,1\}$, of the unit sphere. By the definition \eqref{201509161237}
of the Lissajous curve $\Lissajous$, the identity $\Lissajous(t)=(0,0, v)$
holds if and only if
\begin{equation*}
m_2 t \eqsim \frac{1-v}{2} \pi \quad \text{for $v\in \{-1,1\}$},
\end{equation*}
i.e., if and only if $t = t^{(\vect{m})}_{l}$ with $l\in \{0,m_1, 2m_1, \ldots ,(2m_2-1)m_1\}$. Further,
we have in this case $\Lissajous(t^{(\vect{m})}_{l})=(0,0, 1)$ if $l/m_1$ is even and
$\Lissajous(t^{(\vect{m})}_{l})=(0,0, -1)$ if $l/m_1$ is odd. This yields the statements (i) and (i)' of Proposition \ref{prop-11}.

We consider now the case that $\Lissajous(t)$ is not one of the poles.
By the definition \eqref{201509161237} of  the curve $\Lissajous(t)$, we have $\Lissajous(s) = \Lissajous(t)$ if and only if
\[\cos(m_2 s) = \cos(m_2 t) \quad \text{and} \quad \sin(m_2 s) \left( \begin{array}{l} \cos(m_1 s - \alpha \pi) \\ \sin(m_1 s - \alpha \pi) \end{array}  \right) = \sin(m_2 t)  
 \left( \begin{array}{l} \cos(m_1 t - \alpha \pi) \\ \sin(m_1 t - \alpha \pi) \end{array}  \right).\]
In the first formula we get equality precisely if $m_2 s  \eqsim v m_2 t$ for some $v \in \{-1,1\}$. Plugging this relation into the second formula, we see that
$\sin(m_2 s) = v \sin(m_2 t)$ and, thus, that we get equality in the second formula exactly if $m_1 s \eqsim m_1 t + \frac{1-v}{2} \pi$. In total, we can conclude
that $s\in \mathcal{A}(t)$ if and only if  
\begin{equation}\label{170210-2}
m_2( s - v t) \eqsim 0 \quad \text{and} \quad m_1 ( s - t) + \frac{1-v}{2} \pi \eqsim 0 \qquad \text{for $v\in \{-1,1\}$}
\end{equation}
is satisfied . Since $m_1$ and $m_2$ are relatively prime, B\'ezout's lemma gives two integers
$a, b\in\mathbb{Z}$ such that $a m_1 + b m_2 = 1$. The two conditions
in \eqref{170210-2} imply
\begin{equation} \label{170210}
s \eqsim t - b m_2 (1-v) t - a \frac{1-v}{2} \pi,
\end{equation}
and, thus, $s \eqsim t$ for $v = 1$. For $v = -1$, the two conditions in \eqref{170210-2} imply
\begin{align*} 2 m_1 m_2 s\eqsim  m_1 m_2 (s + t) + m_2 m_1 ( s - t) \eqsim m_2 \pi, \\
2 m_1 m_2 t \eqsim  m_1 m_2 (s + t) - m_2 m_1 ( s - t)  \eqsim m_2 \pi. \end{align*}
Thus, if $m_2$ is even, we have $s = t_{l'}^{(\vect{m})}$ and $t = t_{l}^{(\vect{m})}$ for some $l,l' \in \Zz$. On the other hand, if $m_2$ is odd,
we obtain $s = t_{l'+\frac12}^{(\vect{m})}$ and $t = t_{l+\frac12}^{(\vect{m})}$.

If $m_2$ is even, we can conclude the following: If $t \in [0, 2\pi)$ and $t \neq t_{l}^{(\vect{m})}$ for some $l \in \{0, \ldots, 2 m_1 m_2 -1\}$, then
$t$ is the only element of $[0,2 \pi)$ in $\mathcal{A}^{(\vect{m})}(t)$ and the statement (iii) of the proposition is proven. If
$t \in [0, 2\pi)$ and $t = t_{l}^{(\vect{m})}$, $l\in \{0,\ldots,2m_1m_2-1\}$ and $l \not\equiv 0 \mod m_1$, then $s \in [0,2\pi)$ given
by \eqref{170210} with $v = -1$ satisfies both conditions in \eqref{170210-2} for $v = -1$. Further, because of
the second condition in \eqref{170210-2}, $s \not\eqsim t$. Note that
the particular choice of the numbers $a$ and $b$ from B\'ezout's lemma does not influence equation \eqref{170210} so that if $\vect{m}$ is fixed
and $v = -1$, then $s$ is uniquely determined by $t$. Since $t=t_{l}^{(\vect{m})}$, the so constructed
$s$ can also be written as $s = t_{l'}^{(\vect{m})}$ with some
$l'\in \{0,\ldots,2m_1m_2-1\}$, $l' \not\equiv 0 \mod m_1$ and $l' \neq l$. In total, we can conclude that $\mathcal{A}^{(\vect{m})}(t) = \{s,t\}$
and, thus, the statement (ii) of the proposition. If $m_2$ is odd, we obtain statements (ii)' and (iii)' in an analogous way. \qed
\end{proof}

From the findings in Proposition \ref{prop-11} we see that an even frequency number $m_2$ leads to a slightly different setup of
intersection points than an odd $m_2$. In this article, we will focus on the case that $m_2$ is an even number. In this
case the nodes
\begin{equation} \label{1709171731}
 \LSm_{\alpha} = \left\{\,\Lissajous(t^{(\vect{m})}_{l})\,|\,  l\in \{0,\ldots,2m_1m_2-1\} \,\right\}.
\end{equation}
contain the two poles of the sphere and give a simple characterization of all self-intersection points of the Lissajous curve $\Lissajous$.

\begin{corollary} \label{cor-1}
Let $\mathrm{gcd}(\vect{m}) = 1$ and $m_2$ even. Then, $\LSm_\alpha$ is the set of all self-intersection points of the closed curve
$\Lissajous(t)$, $t \in [0,2\pi)$. $\LSm_\alpha$ contains
$m_2 (m_1 -1) + 2$ points on the sphere $\Sd$, including both poles that are traversed $m_2$ times, and $m_2(m_1-1)$ non-polar double points that
are traversed $2$ times by the curve $\Lissajous(t)$ as $t$ varies from $0$ to $2\pi$.
\end{corollary}

\section{Characterizing spherical Lissajous nodes} \label{sec:nodesphere}

In addition to the description given in Corollary \ref{cor-1}, we can characterize the intersection points of the Lissajous curves also as
the union of two interlacing rectangular grids in spherical coordinates. The construction for this second characterization can be performed for general frequencies
$\vect{m} = (m_1, m_2) \in\Nn^{2}$ where $m_2$ is even. If $m_1$ and $m_2$ are not relatively prime the so obtained nodes can also be interpreted in terms
of Lissajous curves. This relation will be discussed at the end of this section.

\begin{figure}[htb]
	\centering
	\subfigure[\hspace*{1em} The index sets $\I^{(7,6)}$, $\I^{(7,6)}_0$, $\I^{(7,6)}_1$ and $\I^{(7,6)}_{\mathrm{S}}$.
	]{\includegraphics[scale=0.8]{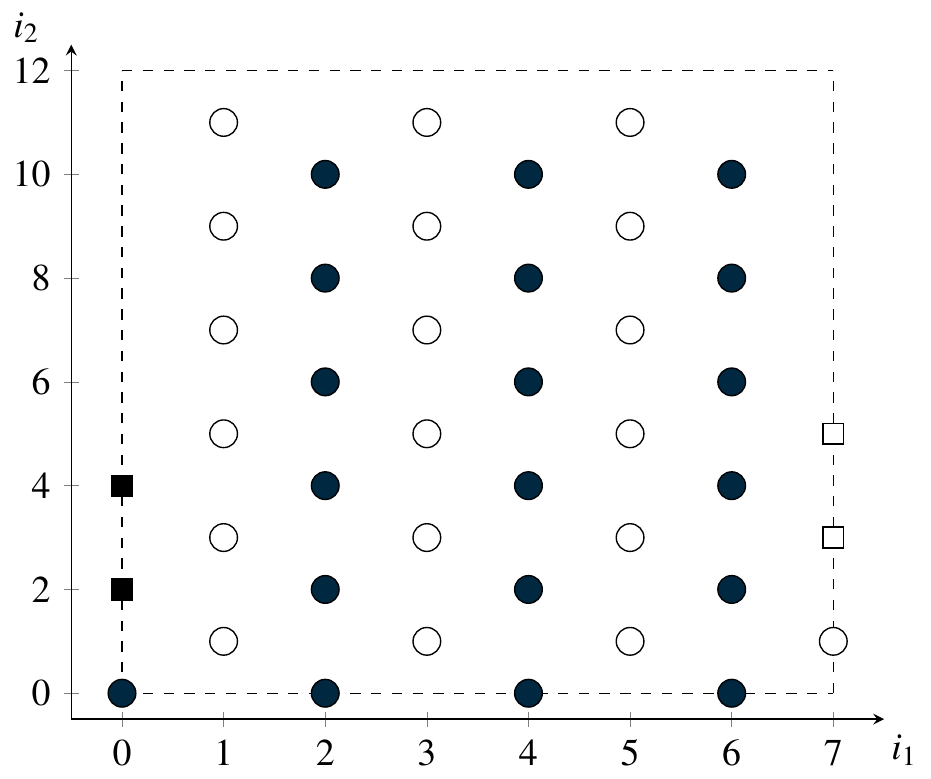}}
	\hfill	
	\subfigure[\hspace*{1em}The index sets $\I^{(6,6)}$, $\I^{(6,6)}_0$, $\I^{(6,6)}_1$ and $\I^{(6,6)}_{\mathrm{S}}$.
	]{\includegraphics[scale=0.8]{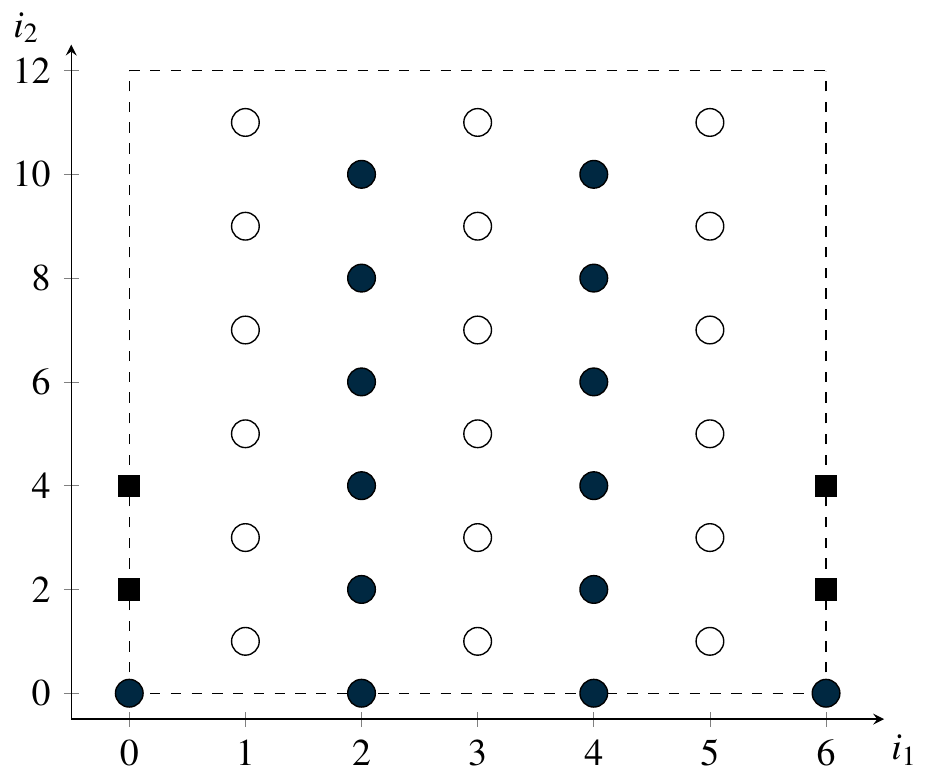}}
  	\caption{Illustration of the index sets $\Immp$, $\Immp_0$ and $\Immp_1$, as well as $\Imms$ as defined in \eqref{eq:0911}, \eqref{eq:0901}, and \eqref{eq:Imms}, 
  	respectively. The black marks describe $\Immp_0$, the white marks $\Immp_1$, the union gives $\Immp$.
  	According to \eqref{eq:09172}, the indices $\vect{i} \in \Immp$ with $i_1 = 0$ and $i_1 = m_1$ describe the north and south pole of $\Sd$, respectively. 
  	For a one to one relation at the poles we consider the reduced subset $\Imms \subset \Immp$. The square marks on the left and right boundary describe the indices 
  	not contained in $\Imms$.
  	} \label{fig:LS-Index}
\end{figure}

To describe the spherical Lissajous nodes we introduce the index set
\begin{equation}\label{eq:0911}
\Immp = \left\{ \,(i_1, i_2) \in \Nn_0^{2}\ \left|\begin {array}{ll}
& 0\leq i_{1}\leq m_{1}, \; 0 \leq i_2 < 2 m_2, \\
& \text{$i_2 < m_2$ \, if $i_{1} \in \{0,m_1\}$}, \\
& \text{$i_{1} + i_{2}$ is even}
\end{array}\right. \, \right\}.
\end{equation}
The set $\Immp$ is a disjoint union $\Immp = \Immp_0 \cup \Immp_1$ of the two sets 
\begin{equation} \label{eq:0901} 
\Immp_0 = \{\vect{i} \in \Immp \ | \ \text{$i_1$, $i_2$ are even} \, \}, \quad
\Immp_1 = \{\vect{i} \in \Immp \ | \ \text{$i_1$, $i_2$ are odd} \, \}.
\end{equation} 

For $\vect{i} = (i_1, i_2) \in \Imm$ we obtain a relation to spherical coordinates by introducing the latitudinal and longitudinal angles
\[ \theta^{(m_1)}_{i_1} = \frac{i_1}{m_1} \pi \in [0,\pi], \qquad \vph^{(m_2)}_{i_2} = \frac{i_2}{m_2} \pi \in [0,2\pi).\]
The set of nodes on the sphere $\Sd$ corresponding to these spherical coordinates is given by
\begin{equation} \label{eq:09172}
\LSm =\left\{\, \vect{x}^{(\vect{m})}_{\vect{i}}\,\left|\,\vect{i}\in \Immp \right.\right\},
\end{equation}
with the points $\vect{x}^{(\vect{m})}_{\vect{i}} \in \Sd$ defined by
\[\vect{x}^{(\vect{m})}_{\vect{i}} = \left(\sin (\theta^{(m_1)}_{i_1}) \cos (\vph^{(m_2)}_{i_2}), \sin (\theta^{(m_1)}_{i_1}) \sin (\vph^{(m_2)}_{i_2}), \cos (\theta^{(m_1)}_{i_1})\right).\]
The cardinality of the set $\Immp$ in \eqref{eq:0911} can be determined from the simple structure of the sets
$\Imm_0$ and $\Imm_1$ in \eqref{eq:0901} (see also Figure \ref{fig:LS-Index}). We have
\begin{equation*}
\#\Immp_{0} = m_1 m_2 / 2,  \quad 
\#\Immp_{1} = m_1 m_2 / 2,
\end{equation*}
and thus 
\[\#\Immp =\#\Immp_0 +\#\Immp_1 = m_1 m_2. \]
All points $\vect{x}^{(\vect{m})}_{\vect{i}}$ with $i_1 = 0$ describe the north pole of $\Sd$ and all
$\vect{x}^{(\vect{m})}_{\vect{i}}$ with $i_1 = m_1$ the south pole. Therefore, the cardinality
of $\LSm$ is smaller than $\#\Imm$. A simple counting gives $\#\LSm = (m_1 -1) m_2 +2$.
In order to have a one to one correspondence between indices and elements in $\LSm$ we introduce the following subsets of $\Imm$:
\begin{equation} \label{eq:Imms}
 \Imms = \{\vect{i} \in \Immp \ | \ i_2 \in \{0,1\} \ \text{if} \ i_1 \in \{0, m_1\}\}. 
\end{equation}
Clearly $\Imms \subset \Immp $ and  
$\# \Imms = \# \LSm = (m_1 -1) m_2 +2$.
Moreover, every node in $\LSm$ can now be
described in a unique way by an index $\vect{i} \in \Imms$. In particular, we have
\[\LSm =\left\{\, \vect{x}^{(\vect{m})}_{\vect{i}}\,\left|\,\vect{i}\in \Imms \right.\right\}.\]

As a basis for the interpolation on the sphere, we will use a double Fourier basis that is not continuous at the poles of the sphere. It makes therefore sense 
to formulate the interpolation theory first in terms of the larger index set $\Immp$. In a second step, we will then reduce the problem to the subset $\Imms$ and the corresponding Lissajous node points $\LSm$. The reason for the halved number of elements at the left and right boundary in $\Imm$ is a glide reflection symmetry of
the used Fourier basis. This symmetry will play an important role when we discuss the implementation of the interpolation scheme.

The following more technical result provides an identification of the index set $\Immp$
with a class decomposition of the product set $H^{(\vect{m})}\times {R}^{(\vect{m})}$, where the sets $H^{(\vect{m})}$ and ${R}^{(\vect{m})}$ are given as
\[ H^{(\vect{m})}=\{0,\ldots, 2 m_{1} m_{2}/g -1\} \qquad
\text{and}\qquad
R^{(\vect{m})}= \{0,\ldots,g-1\}.\]
Here, $g = \mathrm{gcd}(\vect{m})$ denotes again the greatest common divisor of the integers $m_1$ and $m_2$. This result will provide us the 
link between the nodes $\LSm$ and the involved generating Lissajous curves.

\begin{proposition}\label{1509221521} Let $\vect{m} \in \Nn^2$ and $m_2$ be even. \vspace{-1mm}
\begin{itemize}
\item[a)] For all\, $(l,\rho)\in H^{(\vect{m})}\times R^{(\vect{m})}$, there exists an $\vect{i}\in \Immp$ and a $v \in\{-1,1\}$ such that
\begin{align}
i_{1} &\equiv v l \mod 2 m_1, \label{1509221526} \\
i_{2} &\equiv l - 2\rho \frac{m_2}{g} - \frac{1-v}{2} m_2 \mod 2 m_2. \label{1509221526B}
\end{align}
The number $v \in\{-1,1\}$ and the element $\vect{i}\in \Immp$ are uniquely determined by \eqref{1509221526} and \eqref{1509221526B}.
In this way, a function $\vect{i}^{(\vect{m})}:\,H^{(\vect{m})}\times R^{(\vect{m})}\to\Immp$ is well defined by
\begin{equation*}
\vect{i}^{(\vect{m})}(l,\rho)=\vect{i}.
\end{equation*} \vspace{-8mm}
\item[b)] $\vect{i}^{(\vect{m})}(l,\rho)\in \Imm_{0}$ if and only if $l$ is even, and $\vect{i}^{(\vect{m})}(l,\rho)\in \Imm_{1}$ 
if and only if $l$ is odd.
\item[c)] For all
$\vect{i}\in \Immp$ we have $\#\{\,(l,\rho)\in H^{(\vect{m})}
\times R^{(\vect{m})}\,|\,\vect{i}^{(\vect{m})}(l,\rho)=\vect{i}\,\} = 2$.
\end{itemize}
\end{proposition}

\begin{proof}
We start with statement a). For $l \in H^{(\vect{m})}$ we can find an integer $0\leq i_{1}\leq m_{1}$ and
a $v \in \{-1,1\}$ such that \eqref{1509221526} is satisfied. Clearly, the number $i_1$ is uniquely determined by this condition
whereas $v \in \{-1,1\}$ is only uniquely determined if $l \not\equiv 0 \mod 2 m_1$ and $l \not\equiv m_1 \mod 2 m_1$.
Further, for $(l,\rho)$ and $v \in \{-1,1\}$ given by \eqref{1509221526}, equation \eqref{1509221526B} gives a uniquely determined integer
$0\leq i_{2}< 2m_{2}$ in the case that $l$ is not divisible by $m_1$ or $2 m_1$. In the case that $l \equiv 0 \mod 2 m_1$ or $l \equiv m_1 \mod 2 m_1$,
condition \eqref{1509221526B} provides a unique integer $0\leq i_{2}< m_{2}$ by determining at the same time the value of $v \in \{-1,1\}$.
Since $m_2$ is even, we have $i_{1} \equiv l\equiv  i_{2} \mod 2$.
This implies $\vect{i}\in \Immp$ and, therefore, statement a). Statement b) follows also directly from \eqref{1509221526}, \eqref{1509221526B}
and the definition in \eqref{eq:0901}.

We finally turn to statement c). If $\vect{i}\in\Immp$ and $v\in \{-1,1\}$, then the integers
$a_{1}=v i_1$ and $a_2 = i_2 + m_2 (1-v)/2$ are uniquely determined by $\vect{i}\in\Immp$, $v\in \{-1,1\}$ and
satisfy $a_1 \equiv a_2 \mod 2$. Since $m_1$ and $m_2/g$ are relatively prime the Chinese remainder theorem 
yields a unique number $l\in \{0,\ldots,2m_1m_2/g-1\}$ such that
\[
 v i_{1} \equiv l \mod 2m_1, \qquad i_{2} + m_2 (1-v)/2 \equiv l \mod 2m_2/g.
\]
Now, we can find also a uniquely determined $\rho\in {R}^{(\vect{m})}$ such that \eqref{1509221526B} holds. Thus, since both choices of $v$ give
distinct elements $(l,\rho)$, statement c) is shown. \qed
\end{proof}

\begin{figure}[htb]
	\centering
	\subfigure[The Lissajous nodes $\LS^{(6,6)}$ and the union $ \bigcup_{\rho = 0}^{5 }\vect{\ell}^{(6,6)}_{\rho/3}$.
	The red curve describes $\ell^{(6,6)}_{2/3}$.
	]{\includegraphics[scale=0.9]{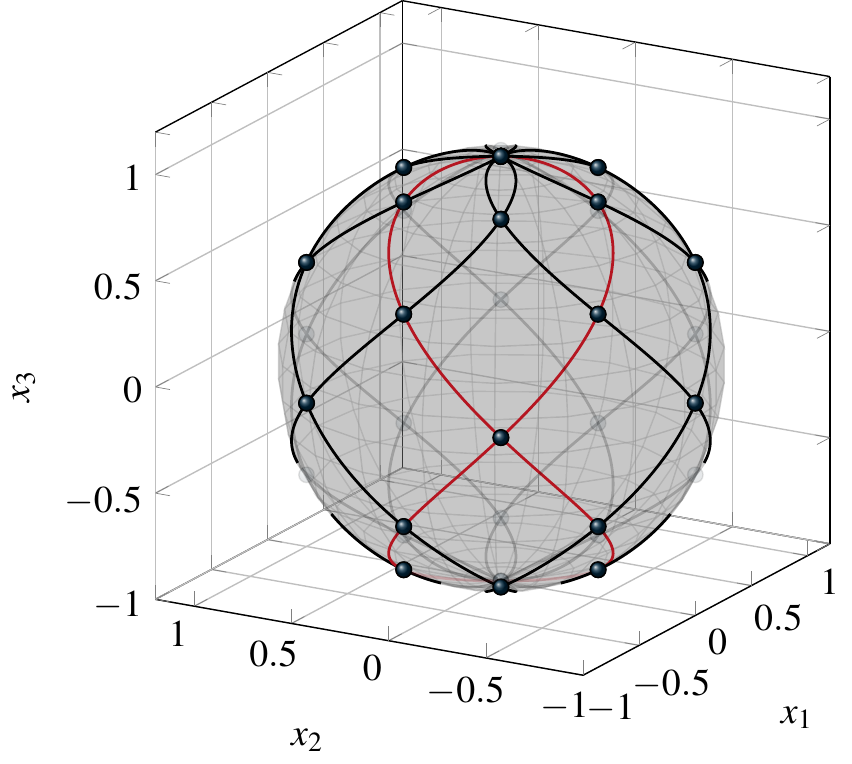}}
	\hfill	
	\subfigure[The Lissajous nodes $\LS^{(4,4)}$ and the union $ \bigcup_{\rho = 0}^{3 }\vect{\ell}^{(4,4)}_{\rho/2}$.
	The red curve describes $\ell^{(4,4)}_{1/2}$.
	]{\includegraphics[scale=0.9]{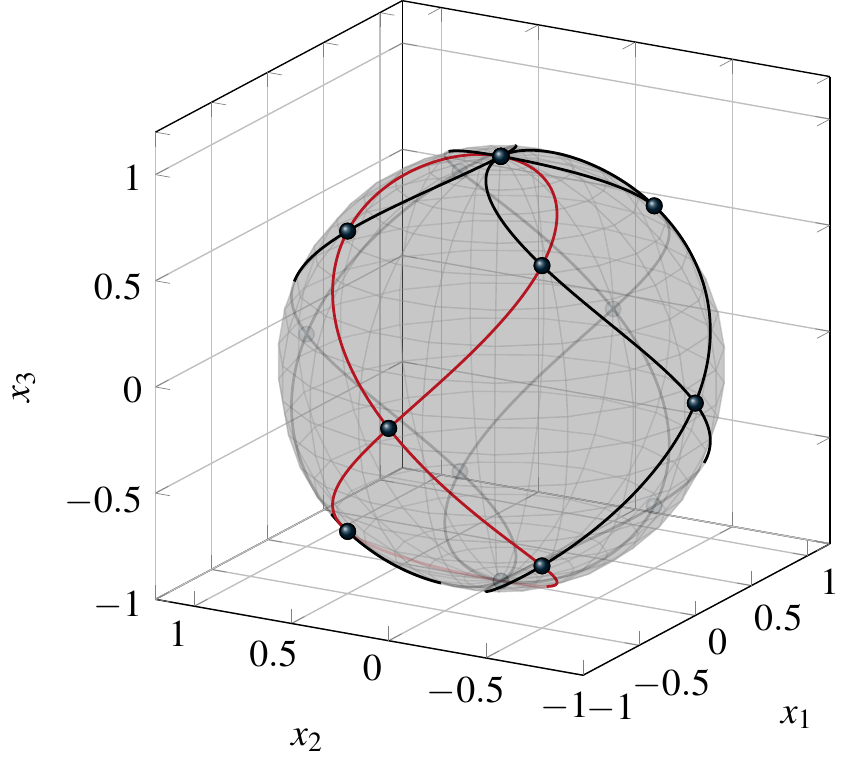}}
  	\caption{Illustration of the nodes $\LS^{(\vect{m})}$ and the union of their generating curves for $m_1 = m_2$.
  	} \label{fig:LS-2}
\end{figure}

A simple consequence of this proposition is the following description of the nodes $\LSm$. 

\begin{corollary} \label{cor-2}
Let $\vect{m} \in \Nn^2$ and $m_2$ even. Then
\[\LSm = \bigcup_{\rho = 0}^{g-1} \left\{\,\vect{\ell}^{(\vect{m})}_{2 \rho/m_2} (t^{(\vect{m})}_{l})\ | \  l\in \{0,\ldots,2m_1m_2/g-1\} \,\right\},\]
where $t^{(\vect{m})}_{l}$ are the equidistant sampling points given in \eqref{eq-samples1} and $\Lissajous$ the Lissajous curves introduced in \eqref{201509161237}.
In particular, if $m_1$ and $m_2$ are relatively prime, then $\LSm$ corresponds to the self-intersection points $\LSm_0$ of the curve
$\vect{\ell}^{(\vect{m})}_{0}$ given in \eqref{1709171731}. 
\end{corollary}

\begin{proof}
By definition of the Lissajous curve and the sampling points $t^{(\vect{m})}_{l}$, we have
\[ \ts \vect{\ell}^{(\vect{m})}_{2 \rho/m_2} (t^{(\vect{m})}_{ l}) =
\left( \sin \left( \frac{ l \pi}{m_1} \right) \cos \left( \frac{l \pi}{m_2} - \frac{2 \rho}{m_2} \pi \right), \ \sin \left(\frac{ l \pi}{m_1} \right)
\sin \left(\frac{l \pi}{m_2} - \frac{2 \rho}{m_2} \pi \right), \
\cos \left( \frac{ l \pi}{m_1}\right)  \right).\]
Now applying Proposition \ref{1509221521} we can find $\vect{i}\in \Imm$ and $v \in\{-1,1\}$ such that the relations \eqref{1509221526} and \eqref{1509221526B} are
satisfied. In particular, this implies
\begin{align*} \vect{\ell}^{(\vect{m})}_{2 \rho /m_2} (t^{(\vect{m})}_{ l}) &=  \ts
\left( \sin \left( \frac{ v i_1 \pi}{m_1} \right) \cos \left( \frac{ i_2 \pi}{m_2} -\frac{1-v}{2} \pi \right), \ \sin \left( \frac{ v i_1 \pi}{m_1} \right)
\sin \left(\frac{ i_2 \pi}{m_2} -\frac{1-v}{2} \pi \right), \
\cos \left(  \frac{ v i_1 \pi}{m_1}\right)  \right) \\
&=  \ts
\left( \sin \left( \frac{ i_1 \pi}{m_1} \right) \cos \left( \frac{ i_2 \pi}{m_2} \right), \ \sin \left( \frac{ i_1 \pi}{m_1} \right)
\sin \left(\frac{ i_2 \pi}{m_2} \pi \right), \
\cos \left(  \frac{ i_1 \pi}{m_1}\right)  \right) = \vect{x}_{\vect{i}}^{(\vect{m})}.
\end{align*}
Going these steps back, we get the reverse implication:
if $\vect{x}_{\vect{i}}^{(\vect{m})}$ is given, we can fix $v \in\{-1,1\}$ and obtain by Proposition \ref{1509221521} a unique pair $(l,\rho)$ such that
$\vect{x}_{\vect{i}}^{(\vect{m})} = \vect{\ell}^{(\vect{m})}_{2 \rho /m_2} (t^{(\vect{m})}_{ l})$. \qed
\end{proof}

If $m_1$ and the even $m_2$ are relatively prime, Corollary \ref{cor-2} provides the second
attempt to characterize the self-intersection points of the spherical Lissajous curves mentioned at the beginning of this section. If $m_1$ and $m_2$ are not
relatively prime, it states that $\LSm$ can be generated by time equidistant samples of at most $g$ different Lissajous curves. 
Two examples of node sets $\LSm$ in which $m_1$ and
$m_2$ are not relatively prime are illustrated in Figure \ref{fig:LS-2}.

\section{Discrete orthogonality structure on $\Immp$}\label{1507091240}

We denote by $\mathcal{L}(\Immp)$ the space of all discrete functions on $\Immp$. For $\vect{\gamma}\in \Zz^{2}$, we consider
the functions $\polbas\in \mathcal{L}(\Immp)$ given by
\begin{equation}\label{A1508291531}
\polbas(\vect{i}) = \left\{ \begin{array}{ll} \cos(\gamma_{1} i_1 \pi/m_{1}) \mathrm{e}^{\imath \gamma_{2} i_2 \pi/m_{2} } & \text{if $\gamma_2$ is even}, \\                                              \imath \sin(\gamma_{1} i_1 \pi/m_{1}) \mathrm{e}^{\imath \gamma_{2} i_2 \pi/m_{2} } & \text{if $\gamma_2$ is odd}.
\end{array} \right.
\end{equation}
The functions $\polbas$ are a discretization of the parity-modified Fourier basis that we will discuss in the next section. The goal
of this section is to establish a discrete orthogonality of the functions $\polbas$ on $\Immp$ similar to the discrete 
orthogonality structure developed for the Lissajous-Chebyshev points in \cite{DenckerErb2017a,DenckerErb2015a}. This will be our main technical prerequisite
for the proofs of the upcoming interpolation results.

We denote the normalized uniform discrete measure on the power set of $\Immp$ by $\omega^{(\vect{m})}$. It is determined by
$\omega^{(\vect{m})}(\{\vect{i}\})= 1/(m_1m_2)$. The vector space $\mathcal{L}(\Immp)$ endowed
with the inner product
\[ \langle f,g \rangle_{\omega^{(\vect{m})}} = \int f \, \overline{g} \, \mathrm{d}\rule{1pt}{0pt}\omega^{(\vect{m})}
= \frac1{m_1m_2}\sum_{\vect{i} \in \Immp} f(\vect{i}) \overline{g(\vect{i})}\]
is a Hilbert space. The corresponding norm is denoted by $\|\cdot\|_{\omega^{(\vect{m})}}$.

\begin{proposition}\label{1507211320}
Let $\vect{m} \in \Nn^2$, $m_2$ be even and $\vect{\gamma}\in\Zz^{2}$. If $\int\polbas \mathrm{d}\rule{1pt}{0pt}\omega^{(\vect{m})}\neq 0$, then
\begin{equation}\label{1507201132}
\text{there exist $(h_1,h_2) \in \Zz_0^{2}$ with $\gamma_{1}=h_{1}m_{1}$, $\gamma_{2}=h_{2}m_{2}$, and
$h_1 + h_2 \equiv 0 \mod 2$}.
\end{equation}
If \eqref{1507201132} is satisfied, then $\int\polbas \mathrm{d}\rule{1pt}{0pt}\omega^{(\vect{m})}=1$.
\end{proposition}

In the proof, we use for $N\in\Nn_0$ the well-known trigonometric identity
\begin{equation}\label{1506171253}
\sum_{l=0}^N \mathrm{e}^{\imath l \vartheta} = \left\{ \begin{array}{ll} \frac{\mathrm{e}^{\imath (N+1) \vartheta}-1}{\mathrm{e}^{\imath \vartheta}-1}
\quad & \vartheta\notin 2\pi\mathbb{Z},\\ N+1 & \vartheta \in 2\pi\mathbb{Z}. \end{array} \right.
\end{equation}

\begin{proof}
We start with the case $\gamma_2 \equiv 0 \mod 2$. Then,
using Proposition \ref{1509221521}, we obtain
\begin{align*}
\int\polbas\mathrm{d}\rule{1pt}{0pt}\omega^{(\vect{m})}
&=  \frac{1}{m_1m_2}\sum_{\vect{i}\in\Immp}
 \cos(\gamma_{1} i_1 \pi/m_{1}) \mathrm{e}^{\imath \gamma_{2} i_2 \pi/m_{2} } \\
&=  \frac{1}{2m_1m_2} \sum_{v \in \{-1,1\}}\sum_{\vect{i}\in\Immp}
\mathrm{e}^{\imath \,( \gamma_{1} v i_1 \pi/m_{1} + \gamma_{2} i_2 \pi/m_{2})}\\
&=  \frac{1}{2m_1m_2} \sum_{l \in H^{(\vect{m})}} \sum_{\rho \in R^{(\vect{m})}}
\mathrm{e}^{\imath \,( \gamma_{1} l \pi/m_{1} + \gamma_{2} l \pi/m_{2} + 2 \gamma_2 \rho \pi / m_2)}.
\end{align*}
In view of \eqref{1506171253}, this integral is only different from zero if $\gamma_{1} /m_{1} + \gamma_{2} /m_{2} \in 2 \Zz$ and $\gamma_2 /m_2 \in \Zz$
are satisfied. Thus, if we assume that the integral $\int\polbas\mathrm{d}\rule{1pt}{0pt}\omega^{(\vect{m})} \neq 0$ then $\gamma_2 = h_2 m_2$ with some
integer $h_2 \in \Zz$ and $\gamma_{1} /m_{1} + \gamma_2 / m_2 \in 2 \Zz$. Thus, also $\gamma_1$ is of the form $\gamma_1 = h_1 m_1$ with an integer
$h_1 \in \Zz$ and we further have $h_1 + h_2 \in 2 \Zz$.
This proves \eqref{1507201132}. On the other hand if \eqref{1507201132} is satisfied then \eqref{1506171253} gives
\begin{align*}
\int\polbas\mathrm{d}\rule{1pt}{0pt}\omega^{(\vect{m})}
&=  \frac{1}{2m_1m_2} \sum_{l \in H^{(\vect{m})}} \mathrm{e}^{\imath \,l ( \gamma_{1}  \pi/m_{1} + \gamma_{2} \pi/m_{2})} \sum_{\rho \in R^{(\vect{m})}}
\mathrm{e}^{\imath \,( 2 \gamma_2 \pi / g) \rho } = 1.
\end{align*}

In the case $\gamma_2 \equiv 1 \mod 2$, we obtain with Proposition \ref{1509221521}
\begin{align*}
\int\polbas\mathrm{d}\rule{1pt}{0pt}\omega^{(\vect{m})}
&=  \frac{1}{m_1m_2} \sum_{\vect{i}\in\Immp}
 \imath \sin(\gamma_{1} i_1 \pi/m_{1}) \mathrm{e}^{\imath \gamma_{2} i_2 \pi/m_{2} } \\
&=  \frac{1}{2m_1m_2} \sum_{v \in \{-1,1\}}\sum_{\vect{i}\in\Immp}
\mathrm{e}^{\imath \,( \gamma_{1} v i_1 \pi/m_{1} + \gamma_{2} i_2 \pi/m_{2} + (1-v)\pi /2 )}\\
&=  \frac{1}{2m_1m_2} \sum_{l \in H^{(\vect{m})}} \sum_{\rho \in R^{(\vect{m})}}
\mathrm{e}^{\imath \,( \gamma_{1} l \pi/m_{1} + \gamma_{2} l \pi/m_{2} + 2 \gamma_2 \rho \pi /m_2)}.
\end{align*}
Now, with the same argumentation as above the fact that this integral does not vanish implies the conditions
$\gamma_1 = h_1 m_1$, $\gamma_2 = h_2 m_2$ for some integers $h_1$ and $h_2$, as well as $h_1 + h_2 \in 2 \Zz$.
Furthermore, if \eqref{1507201132} is satisfied also in this case we get
$ \int\polbas\mathrm{d}\rule{1pt}{0pt}\omega^{(\vect{m})} = 1$. \qed
\end{proof}

\begin{figure}[htb]
	\centering
	\subfigure[The spectral index set $\vect{\Gamma}^{(7,6)} = \overline{\vect{\Gamma}}^{(7,6)}$.
	]{\includegraphics[scale=0.79]{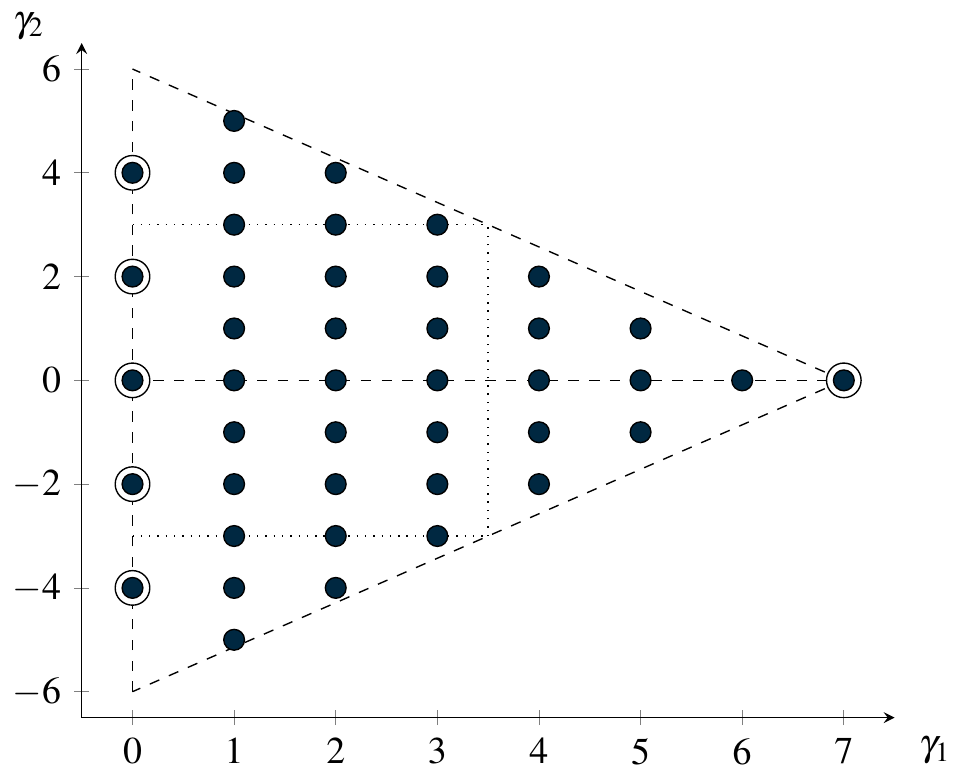}}
	\hfill	
	\subfigure[The spectral index sets $\vect{\Gamma}^{(6,6)}$ and $\overline{\vect{\Gamma}}^{(6,6)}$.
	]{\includegraphics[scale=0.79]{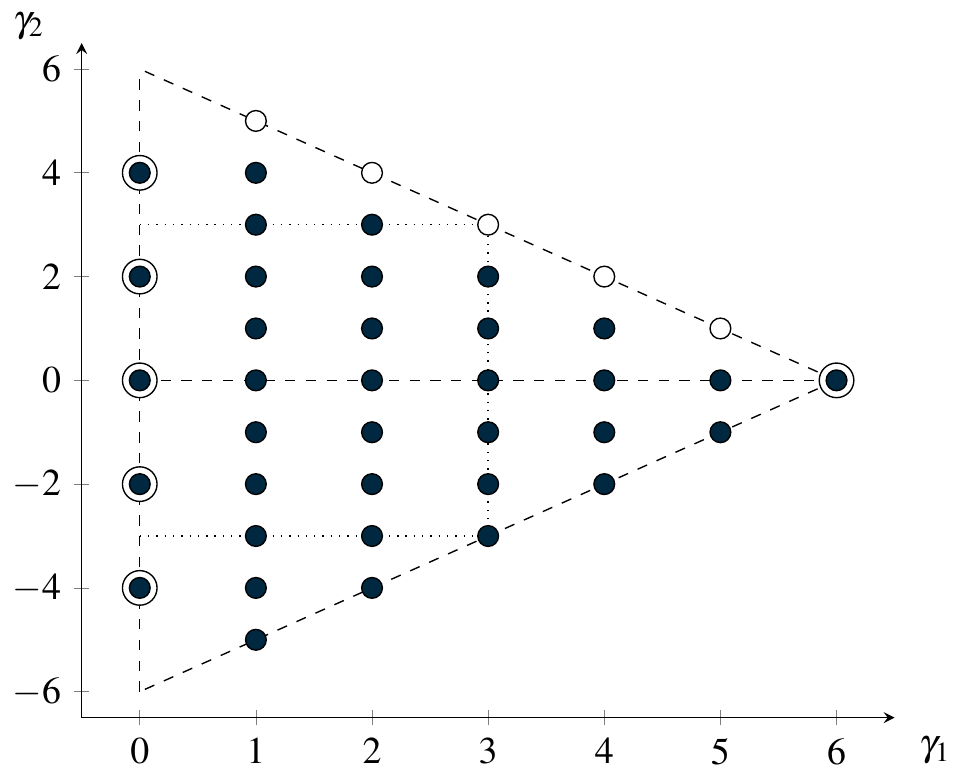}}
  	\caption{Illustration of the index sets $\Gam$ (black dots) and $\Gamsup$ (black and white dots).
  	The white dots are the elements of $\Gamup$. The circled dots
  	indicate the basis functions in \eqref{1508221825} with doubled norm.
  	} \label{fig:LS-3}
\end{figure}

Using the discrete functions $\polbas$ and Proposition \ref{1507211320}, we are now going to construct
two orthogonal basis systems in the space $\mathcal{L}(\Immp)$.
For this, we introduce the spectral index set
\begin{equation}\label{1508222042sup}
\Gamsup = \left\{\,\vect{\gamma}\in \Nn \times \Zz \left|\begin {array}{ll}
& 1 \leq \gamma_{1} \leq m_{1},\\
& \gamma_{1}/m_{1}+ |\gamma_{2}|/m_{2} \leq 1
\end{array}\right.
\right\} \cup \left\{\,(0,\gamma_2) \left|\begin {array}{ll}
& \gamma_2 \in 2 \Zz, \\
& |\gamma_{2}| < m_{2}
\end{array}\right.
\right\}.
\end{equation}
For odd $\gamma_2$, we have $\chi^{(\vect{m})}_{(0,\gamma_2)}(\vect{i}) = 0$ for all $\vect{i} \in \Immp$. Hence, these indices are excluded in \eqref{1508222042sup}. 
The index set $\Gamsup$ is in general still too large for our purpose. 
Some of the functions $\polbas$, $\vect{\gamma} \in \Gamsup$, are linearly dependent in $\mathcal{L}(\Immp)$.
This linear dependence in $\Gamsup$ is related to the two sets
\begin{align*}
\Gamup &= \left\{ \vect{\gamma}\in \Gamsup \ | \  \gamma_{1}/m_{1} + \gamma_{2}/m_{2} = 1, \ \gamma_2 \neq 0 \ \right\}, \\
\Gamdown &= \left\{ \vect{\gamma}\in \Gamsup \ | \ \gamma_{1}/m_{1} - \gamma_{2}/m_{2} = 1, \ \gamma_2 \neq 0 \ \right\}.
\end{align*}
In particular, the set given by
\begin{equation}\label{1508222042B}
\Gam = \Gamsup \setminus \Gamup .
\end{equation}
will soon turn out to be the right spectral index set for our considerations. Note that the choice of $\Gamup$ over $\Gamdown$ in \eqref{1508222042B} is
arbitrary and can be also switched for the subsequent results. Also note that if $m_1$ and $m_2$ are relatively prime then $\Gamup = \Gamdown = \emptyset$ and 
$\Gam = \Gamsup$. A simple counting argument
gives the following complexities:
\[ \begin{array}{rlllll}
\#\Gamsup &=&  m_1 m_2 + g - 1,  \qquad & \#\Gamup &=& g-1, \\
\#\Gam  &=&  m_1 m_2, & \#\Gamdown &=& g-1.
\end{array} \]
For the proof of the subsequent theorem, we will use the following product formulas:
\begin{align}
\label{1507222159}
\polbas \overline{\chi^{(\vect{m})}_{\vect{\gamma}'}}
&= \frac{1}{2}  \left(\chi^{(\vect{m})}_{(\gamma_1-\gamma_1',\gamma_2 - \gamma_2')}
+ (-1)^{\gamma_2'} \chi^{(\vect{m})}_{{(\gamma_1+\gamma_1',\gamma_2 - \gamma_2')}} \right),\\
\polbas \chi^{(\vect{m})}_{\vect{\gamma}'}
&= \frac{1}{2}  \left(\chi^{(\vect{m})}_{(\gamma_1+\gamma_1',\gamma_2 + \gamma_2')}
+ (-1)^{\gamma_2'} \chi^{(\vect{m})}_{{(\gamma_1-\gamma_1',\gamma_2 + \gamma_2')}} \right), \label{1507222159B}\\
\overline{\chi^{(\vect{m})}_{\vect{\gamma}}}
&= (-1)^{\gamma_2} \chi^{(\vect{m})}_{(\gamma_1,-\gamma_2)}. \label{1507222159C}
\end{align}
These identities are satisfied for all $\vect{\gamma},\vect{\gamma}' \in \Zz^2$ and can be derived directly from the definition \eqref{A1508291531}
of the functions $\polbas$ using standard identities for the products of two trigonometric functions.

\begin{theorem}\label{1507091911} Let $\vect{m} \in \Nn^2$, $m_2$ be even.
The functions $\polbas$, $\vect{\gamma}\in \Gam$, form an
orthogonal basis of the $m_1m_2$ dimensional inner product space $(\mathcal{L}(\Immp),\langle\,\cdot,\cdot\,\rangle_{\omega^{(\vect{m})}})$.
The norms of the basis functions $\polbas$ are given as
\begin{equation}\label{1508221825}
 \|\polbas\|_{\omega^{(\vect{m})}}^2 =
 \left\{ \begin{array}{rl} 1,\; & \text{if}\quad \vect{\gamma} \in \Gam, \ \gamma_1 \in \{0,m_1\},\\
 \frac12,\; & \text{if}\quad \vect{\gamma} \in \Gam, \ \gamma_1 \notin \{0,m_1\}.
   \end{array} \right.
\end{equation}

\end{theorem}

\begin{proof}
We will continuously use the product formula \eqref{1507222159} in this proof. Therefore, we denote the index vectors on the right hand side of
\eqref{1507222159} by
\[ \vect{\gamma}^+ = (\gamma_1+\gamma_1',\gamma_2-\gamma_2') \quad \text{and} \quad \vect{\gamma}^- = (\gamma_1-\gamma_1',\gamma_2-\gamma_2').\]
We assume first that $\vect{\gamma},\vect{\gamma}'\in \Gam$ and $\vect{\gamma} \neq \vect{\gamma}'$. We differentiate between two subcases. \\[1mm]
\emph{Case 1:} $\vect{\gamma} = (m_1,0)$ or $\vect{\gamma'} = (m_1,0)$. Without restriction, we assume
that $\vect{\gamma} = (m_1,0)$. Then, we have
\[ m_1 \leq m_1 + \gamma_1' < 2 m_1, \quad 0 < m_1 - \gamma_1' \leq m_1, \quad m_2 < \gamma_2' < m_2.\]
This implies that both $\vect{\gamma}^+$ and $\vect{\gamma}^-$ don't satisfy the condition \eqref{1507201132} and therefore, by \eqref{1507222159}, we obtain
$\int\polbas \overline{\chi^{(\vect{m})}_{\vect{\gamma}'}} \mathrm{d}\rule{1pt}{0pt}\omega^{(\vect{m})} = 0$. \\[2mm]
\emph{Case 2:} $\vect{\gamma} \neq (m_1,0)$ and $\vect{\gamma'} \neq (m_1,0)$. Then, based on our assumptions on $\vect{\gamma}$ and $\vect{\gamma}'$,
we obtain the inequalities
\[ 0 \leq \gamma_1 + \gamma_1' < 2 m_1, \quad -m_1 < \gamma_1 - \gamma_1' < m_1, \quad 2 m_2 < \gamma_2-\gamma_2' < 2 m_2.\]
For $\vect{\gamma}^-$, the condition \eqref{1507201132} can only be satisfied if $\vect{\gamma}=\vect{\gamma}'$, which is excluded by the
given assumptions. For $\vect{\gamma}^+$, the condition \eqref{1507201132} is satisfied if $\vect{\gamma}=\vect{\gamma}'=(0,0)$ or
if $\gamma_1 + \gamma_1' = m_1$ and $|\gamma_2-\gamma_2'| = m_2$ holds true.
The first instance can be excluded by the assumption $\vect{\gamma} \neq \vect{\gamma}'$. Also, the second instance can be excluded, since
by $\Gam \subset \Gamsup$ and $\Gam \cap \Gamup = \emptyset$ we have
\[ \frac{\gamma_1}{m_1} + \frac{\gamma_2}{m_2} + \frac{\gamma_1'}{m_1} - \frac{\gamma_2'}{m_2} < 2, \quad
\frac{\gamma_1}{m_1} - \frac{\gamma_2}{m_2} + \frac{\gamma_1}{m_1} + \frac{\gamma_2}{m_2} < 2. \]
Thus, by Proposition \ref{1507211320} and \eqref{1507222159} we obtain also for the second case
$\int\polbas \overline{\chi^{(\vect{m})}_{\vect{\gamma}'}} \mathrm{d}\rule{1pt}{0pt}\omega^{(\vect{m})} = 0$. In total, we
can conclude that the functions $\polbas$, $\vect{\gamma}\in \Gam$, are pairwise
orthogonal with respect to the inner product $\langle\,\cdot,\cdot\,\rangle_{\omega^{(\vect{m})}}$.\\[2mm]
We now have a look at the norms of the functions $\polbas$ and consider the case $\vect{\gamma} = \vect{\gamma}'$ in \eqref{1507222159}.
We get $\vect{\gamma}^+ = (2 \gamma_1,0)$ and $\vect{\gamma}^- = (0,0)$
Since $\vect{\gamma}\in \Gam$, we have $0 \leq 2 \gamma_1 \leq 2 m$. Therefore, condition \eqref{1507201132} is always satisfied for
$\vect{\gamma}^-$ and satisfied for $\vect{\gamma}^+$ precisely if $\gamma_1 \in \{0,m_1\}$. Based on this observation Proposition \ref{1507211320}
implies \eqref{1508221825}. \\
Finally, since the functions $\polbas$, $\vect{\gamma} \in \Gam$, form an orthogonal system consisting of $m_1 m_2$ elements and $\funpol(\Immp)$ is a space of the same
dimension, the functions $\polbas$, $\vect{\gamma} \in \Gam$, are an orthogonal basis of $\funpol(\Immp)$. \qed
\end{proof}

\begin{remark} Theorem \ref{1507091911} holds also true if we replace $\Gamup$ in the definition \eqref{1508222042B} of $\Gam$ with the counterpart
$\Gamdown$. The respective proof is identical.
\end{remark}

For practical issues it is convenient to have also a real basis for the vector space $\funpol(\Immp)$. For this, we introduce a second
set of basis functions as
\begin{equation}\label{1702291124}
\polbasreal = \! \left\{ \! \begin{array}{ll}
\Re \chi^{(\vect{m})}_{\vect{\gamma}}
                                        & \text{if}  \quad \begin{array}{l} \vect{\gamma} \in \Gam \setminus \Gamdown, \gamma_2 \leq 0, \quad \text{or}\\
                                                                             \vect{\gamma} \in \Gamdown, \gamma_1 \leq m_1/2, \end{array}  \\[5mm]
\Im \chi^{(\vect{m})}_{\vect{\gamma}}
                                        & \text{if}  \quad \begin{array}{l} \vect{\gamma} \in \Gam \setminus \Gamdown, \gamma_2 > 0, \quad \text{or}\\
                                                                             \vect{\gamma} \in \Gamdown, \gamma_1 > m_1/2. \end{array}
\end{array} \right.
\end{equation}

\begin{theorem} Let $\vect{m} \in \Nn^2$, $m_2$ be even.
The functions $\polbasreal$, $\vect{\gamma}\in \Gam$, are a real
orthogonal basis of the inner product space $(\funpol(\Immp),\langle\,\cdot,\cdot\,\rangle_{\omega^{(\vect{m})}})$.
The norms for $\polbasreal$ read as
\begin{equation}\label{1508221826}
 \|\polbasreal\|_{\omega^{(\vect{m})}}^2 =
 \left\{ \begin{array}{rl}     1\; & \text{if}\quad \vect{\gamma} \in \{(0,0), (m_{1},0) \}, \\
                             \frac12\; & \text{if}\quad \vect{\gamma} \in \Gam \setminus \{(0,0),(m_1,0)\}, \; \text{$\gamma_1 = 0$ or $\gamma_2 = 0$},\\[1mm]
                             \frac12\; & \text{if}\quad \vect{\gamma} = (m_1/2,-m_2/2), \\
                             \frac14\; & \text{for all other $\vect{\gamma} \in \Gam$}.
\end{array} \right.
\end{equation}
\end{theorem}

\begin{proof} The functions $\polbasreal$ are clearly all real and given as $\Re \chi^{(\vect{m})}_{(\vect{\gamma})} = \frac12
(\chi^{(\vect{m})}_{\vect{\gamma}} + \overline{\chi^{(\vect{m})}_{\vect{\gamma}}})$ and  $\Im \chi^{(\vect{m})}_{(\vect{\gamma})} = \frac1{2 \imath}
(\chi^{(\vect{m})}_{\vect{\gamma}} - \overline{\chi^{(\vect{m})}_{\vect{\gamma}}})$. Now, based on the formulas \eqref{1507222159}, \eqref{1507222159B}
and \eqref{1507222159C} as well as Proposition \ref{1507211320}, the statements about the orthogonality and the norms
of the basis functions can be derived similarly as in Theorem \ref{1507091911}. As a template for the entire procedure, 
we calculate the norm $\|\polbasreal\|_{\omega^{(\vect{m})}}^2$ for the first case
given in \eqref{1702291124}, i.e., if $\polbasreal = \Re \chi^{(\vect{m})}_{(\vect{\gamma})}$. Here, we get
\begin{align*} \|\polbasreal\|_{\omega^{(\vect{m})}}^2 &= \frac14 \int \left| \chi^{(\vect{m})}_{(\gamma_1, \gamma_2)} + (-1)^{\gamma_2}
\chi^{(\vect{m})}_{(\gamma_1, -\gamma_2)}\right|^2 \mathrm{d}\rule{1pt}{0pt}\omega^{(\vect{m})} \\
&= \frac{1}{8} \int \left( 2 \chi^{(\vect{m})}_{(0, 0)} +2 \chi^{(\vect{m})}_{(2 \gamma_1, 0)} +
\chi^{(\vect{m})}_{(0, 2\gamma_2)} + \chi^{(\vect{m})}_{(0, - 2\gamma_2)} + \chi^{(\vect{m})}_{(2 \gamma_1, 2\gamma_2)}  + \chi^{(\vect{m})}_{(2 \gamma_1, - 2\gamma_2)} \right) \mathrm{d}\rule{1pt}{0pt}\omega^{(\vect{m})},
\end{align*}
where we used the product formulas \eqref{1507222159}, \eqref{1507222159B}
and \eqref{1507222159C} to manipulate the function terms in the integral. Next, we check in which cases the condition \eqref{1507201132} given in Proposition \ref{1507211320} is satisfied
and determine in this way the value of the norm. If $\vect{\gamma} \in \{(0,0), (m_{1},0) \}$, then condition \eqref{1507201132} is satisfied for all six spectral functions in the integral
and we therefore obtain $\|\polbasreal\|_{\omega^{(\vect{m})}}^2 = 1$. If $\vect{\gamma} \in \Gam \setminus \{(0,0),(m_1,0)\}$ and $\gamma_1 = 0$ or $\gamma_2 = 0$ then 
condition \eqref{1507201132} is satisfied only for three of the given basis functions and we obtain $\|\polbasreal\|_{\omega^{(\vect{m})}}^2 = \frac12$. The same holds true
if $\vect{\gamma} = (m_1/2,-m_2/2)$. In the remaining case $\gamma_1>0$, $\gamma_2 > 0$, the condition \eqref{1507201132} is only satisfied for $\chi^{(\vect{m})}_{(0, 0)}$ and we thus 
obtain $\|\polbasreal\|_{\omega^{(\vect{m})}}^2 = \frac14$. \qed
\end{proof}

\section{Interpolation on spherical Lissajous points} \label{sec:interpolation}
We are now ready to set up an interpolation scheme for the Lissajous nodes $\LSm$ on the sphere $\Sd$. We consider general frequencies $\vect{m} \in \Nn^2$ where $m_2$ is even.
For simplicity, we will formulate the interpolation problem in the domain $[0,\pi] \times [0,2\pi)$ of the spherical coordinates $(\theta,\vph)$. By \eqref{eq:09172},
the corresponding nodes in spherical coordinates are given as $(\theta^{(m_1)}_{i_1}, \vph^{(m_2)}_{i_2})$, $\vect{i} \in \Immp$. In case we need a one to
one correspondence for the poles of $\Sd$, we will restrict ourselves to the index set $\Imms \subset \Immp$ defined in \eqref{eq:Imms}.

For $\vect{\gamma} \in \Zz^2$, we introduce now the following basis functions in spherical coordinates $(\theta,\vph) \in [0,\pi] \times [0,2\pi)$:
\begin{equation*}
\polbascont(\theta,\vph) = \left\{ \begin{array}{ll} \cos(\gamma_{1} \theta) \mathrm{e}^{\imath \gamma_{2} \vph } & \text{if $\gamma_2$ is even}, \\
                                             \imath \sin(\gamma_{1} \theta) \mathrm{e}^{\imath \gamma_{2} \vph } & \text{if $\gamma_2$ is odd}.
\end{array} \right.
\end{equation*}
This Fourier type basis for functions on the unit sphere is exactly the basis introduced in \cite{Merilees1973,Orszag1974} and mentioned in the introduction.
In the literature \cite{Boyd2000}, it is referred to as parity-modified Fourier basis.
By $\Pi$, we denote the space spanned by all linear combinations of the functions $\polbascont$, $\vect{\gamma} \in \Zz^2$.

The interpolation problem we want to solve can be stated as follows: for given data values $f \in \funpol(\Immp)$ we want to find a function $ P^{(\vect{m})}_f \in \Pi$ such that
\begin{equation}\label{1508220011}
 P^{(\vect{m})}_f (\theta^{(m_1)}_{i_1}, \vph^{(m_2)}_{i_2}) = f({\vect{i}}) \quad \text{for all}\quad \vect{i} \in \Immp.
\end{equation}
In order that \eqref{1508220011} is uniquely solvable, we have to specify an appropriate subspace of $\Pi$ for
the interpolant $P^{(\vect{m})}_f$. For this, the relation
\begin{equation}\label{1508201411}
\polbascont(\theta^{(m_1)}_{i_1}, \vph^{(m_2)}_{i_2}) = \polbas(\vect{i}), \qquad \vect{\gamma}\in\Zz^{2}, \quad \vect{i}\in \Immp,
\end{equation}
between the double Fourier basis $\polbascont$ and the discrete orthogonal basis $\polbas$ for $\funpol(\Immp)$ plays a crucial role.
This relation \eqref{1508201411} and the results of the previous section motivate the introduction of the interpolation space
\[\Pim = \mathrm{span} \left\{\, \polbascont \,\left|\, \vect{\gamma} \in \Gam \right. \right\}.\]

\begin{example}
For even $m \in \Nn$, consider the frequencies $\vect{m} = (m-1,m)$. Then, $\Pim$ is exactly the space of all parity-modified basis functions
$\polbascont$ of total degree $|\gamma_1| + |\gamma_2| \leq m-1$. Thus, in this case the points $\LSm$ can be considered as a spherical analogue of
the Padua points studied in \cite{BosDeMarchiVianelloXu2006,CaliariDeMarchiVianello2005}. For $\vect{m} = (m,m)$, they are a spherical version of the bivariate Morrow-Patterson-Xu points
introduced and studied in \cite{MorrowPatterson1978,Xu1996}. For general $\vect{m} \in \Nn^2$, $m_2$ even, the theory presented in this paper is a spherical analog of the bivariate interpolation
theory based on the nodes of two-dimensional Lissajous curves studied in \cite{DenckerErb2017a,DenckerErb2015a,Erb2015,ErbKaethnerAhlborgBuzug2015}.

In contrast to the actual work, in the literature usually a tensor-product grid in spherical coordinates is used to construct a spectral interpolation scheme on $\Sd$ based on the parity-modified double Fourier basis $\polbascont$, see \cite{Boyd1978,Ganesh1998,Orszag1974,TownsendWilberWright2016}. The corresponding interpolation spaces are defined as 
$\operatorname{span} \{ \polbascont \ | \ 0 \leq \gamma_1 \leq m_1, \ |\gamma_2| \leq m_2\}$ by using a rectangular spectral index set. Respective variants are also established for bivariate polynomial interpolation and are sometimes referred to as maximal degree spaces. A comparison between different bivariate interpolation spaces related to total degree and maximal degree spaces can be found in the treatise \cite{Trefethen2017}.  
\end{example}

In order to have a one to one correspondence between data values on $\LSm$ and $\Immp$,
we additionally consider the subspaces
\begin{align}
\funpols(\Immp) &= \left\{f \in \mathcal{L}(\Immp) \ | \ f(\vect{i}) \equiv f(\vect{j}) \quad \text{if} \; i_1 = j_1 \in \{0,m_1\} \  \right\}. \notag \\
\Pims &= \left\{P \in \Pim \ | \ P(\theta^{(m_1)}_{i_1}, \vph^{(m_2)}_{i_2})) \equiv P(\theta^{(m_1)}_{j_1}, \vph^{(m_2)}_{j_2})) 
\quad \text{if} \; i_1 = j_1 \in \{0,m_1\} \  \right\}. \label{eq:186181745}
\end{align}
Clearly $\funpols(\Immp) \subset \funpol(\Immp)$ and $\dim \funpols(\Immp) = \# \Imms = \# \LSm = \dim \Pims$. The data functions $f \in \funpols(\Immp)$ are 
constant at the coordinates $i_1 = 0$ and $i_1 = m_1$ corresponding to the poles of the sphere. The space $\funpols(\Immp)$ can therefore be used to describe 
all possible data sets on the Lissajous nodes $\LSm$. The subspace $\Pims \subset \Pim$ gives all elements $P \in \Pim$ such that the data set $p(\vect{i}) = P(\theta^{(m_1)}_{i_1}, \vph^{(m_2)}_{i_2}))$, $\vect{i} \in \Imm$, is contained in $\funpols(\Immp)$. Note that, although $P \in \Pims$ satisfies
this discrete pole condition, the function $P \in \Pims$ is in general not constant on the entire lines $\theta = 0$ and $\theta = \pi$ describing the poles. \\

As a fundamental basis for the interpolation problem \eqref{1508220011}, we introduce for $\vect{i} \in \Immp$ the Lagrange functions
\begin{equation} \label{1508220009}
 L^{(\vect{m})}_{\vect{i}}(\theta,\vph) = \frac{1}{m_1m_2} \sum_{\vect{\gamma} \in \Gam} \frac{\overline{\polbas(\vect{i})}}{\|\polbas\|_{\omega^{(\vect{m})}}^2}
 \polbascont(\theta,\vph),
 \qquad (\theta,\vph) \in [0,\pi] \times [0,2\pi).
\end{equation}
For the subset $\Imms$ defined in \eqref{eq:Imms} we use the related variant
\begin{equation} \label{1508220012}
 L^{(\vect{m})}_{\mathrm{S},\vect{i}} = \left\{ \begin{array}{ll}
                                                L^{(\vect{m})}_{\vect{i}} &  \text{if $\vect{i} \in \Imms$, $i_1 \neq \{0,m_1\}$,} \\
                                                \sum_{\vect{j} \in \Immp: j_1 = i_1} L^{(\vect{m})}_{\vect{j}} & \text{if $\vect{i} \in \Imms$, $i_1 \in \{0,m_1\}$.}
                                                \end{array} \right.
\end{equation}
We can now state our main result.
\begin{theorem} \label{201512131945}
Let $\vect{m} \in \Nn^2$, $m_2$ be even and $f\in \funpol(\Immp)$. The interpolation problem \eqref{1508220011} has a unique solution
in the polynomial space $\Pim$ given by
\begin{equation*} 
P^{(\vect{m})}_{f}(\theta,\vph) = \sum_{\vect{i} \in \Immp} f({\vect{i}}) \, L^{(\vect{m})}_{\vect{i}}(\theta,\vph).
\end{equation*}
The Lagrange functions $L^{(\vect{m})}_{\vect{i}}$, $\vect{i} \in \Immp$, form a basis of the vector space $\Pim$. \\
For $f\in \funpols(\Immp)$, the interpolation problem \eqref{1508220011} has a solution of the form
\begin{equation*}
P^{(\vect{m})}_{f}(\theta,\vph) = \sum_{\vect{i} \in \Imms} f({\vect{i}}) \, L^{(\vect{m})}_{\mathrm{S},\vect{i}}(\theta,\vph).
\end{equation*}
This solution is unique in the subspace $\Pims \subset \Pim$ spanned by the functions $L^{(\vect{m})}_{\mathrm{S},\vect{i}}$, $\vect{i} \in \Imms$.
\end{theorem}

\begin{proof}
For $\vect{j} \in \Immp$, let $\delta_{\vect{j}}(\vect{i}) = \delta_{\vect{i}\vect{j}}$ be the Dirac function on $\Immp$. We consider the system $\delta_{\vect{j}}$, $\vect{j} \in \Immp$,
as an orthogonal basis of the space $\funpol(\Immp)$. By Theorem \ref{1507091911}, $\polbas$, $\vect{\gamma} \in \Gam$, is a second orthogonal basis of $\funpol(\Immp)$
and we can expand the functions $\delta_{\vect{j}}$, $\vect{j} \in \Immp$, as
\[\delta_{\vect{j}}(\vect{i}) = \sum_{\vect{\gamma} \in \Gam} \frac{\langle\;\! \delta_{\vect{j}},\polbas \rangle_{\omega^{(\vect{m})}}}{\|\polbas\|_{\omega^{(\vect{m})}}^2} \polbas(\vect{i}) =
\frac{1}{m_1 m_2} \sum_{\vect{\gamma} \in \Gam} \frac{\polbas(\vect{i}) \overline{\polbas(\vect{j})}}{\|\polbas\|_{\omega^{(\vect{m})}}^2}.\]
Evaluating the Lagrange function $L^{(\vect{m})}_{\vect{j}}$, $\vect{j} \in \Immp$, at the points $(\theta,\vph) = (\theta^{(m_1)}_{i_1}, \vph^{(m_2)}_{i_2})$, $\vect{i} \in \Immp$,
and using the identity \eqref{1508201411}, we obtain
\[L^{(\vect{m})}_{\vect{j}}(\theta^{(m_1)}_{i_1},\vph^{(m_2)}_{i_2}) =
\frac{1}{m_1m_2} \sum_{\vect{\gamma} \in \Gam} \frac{\polbas(\vect{i}) \overline{\polbas(\vect{j})}}{\|\polbas\|_{\omega^{(\vect{m})}}^2}.\]
Thus, $L^{(\vect{m})}_{\vect{j}}(\theta^{(m_1)}_{i_1},\vph^{(m_2)}_{i_2}) = \delta_{\vect{j}}(\vect{i})$ and for $f \in \funpol(\Immp)$ the function
$P^{(\vect{m})}_{f}$ satisfies the interpolation condition $\eqref{1508220011}$. Furthermore, the mapping $f \to P^{(\vect{m})}_{f}$ is
an injective homomorphism from $\funpol(\Immp)$ into the space $\Pim$. Since the dimension $\funpol(\Immp)$ coincides with the dimension of $\Pim$ this homomorphism is indeed
an automorphism and the functions $L^{(\vect{m})}_{\vect{j}}$, $\vect{j} \in \Immp$, form a basis of $\Pim$. Finally, we see that
for $f$ in the subspace $\funpols(\Immp)$ the corresponding function $P^{(\vect{m})}_{f}$ is a linear combination of the Lagrange functions
$L^{(\vect{m})}_{\mathrm{S},\vect{j}}$, $\vect{j} \in \Imms$. Since the dimensions of the two subspaces coincide, we get also uniqueness here. \qed
\end{proof}

As in the discrete case, we want to establish the same result also for a real valued basis. To this end we define for $\vect{\gamma} \in \Gam$ the functions
\begin{equation*}
\polbascontreal(\theta,\vph) = \left\{ \begin{array}{ll} \cos(\gamma_{1} \theta) \cos (|\gamma_{2}| \vph )
                                        &  \text{if}  \quad \begin{array}{l} \vect{\gamma} \in \Gam \setminus \Gamdown, \gamma_2 \leq 0, \gamma_2 \ \text{even} \quad \text{or}\\
                                                                             \vect{\gamma} \in \Gamdown, \gamma_1 \leq m_1/2, \gamma_2 \ \text{even}, \end{array}  \\[5mm]
                                 \sin(\gamma_{1} \theta) \sin (|\gamma_{2}| \vph )
                                        &  \text{if}  \quad \begin{array}{l} \vect{\gamma} \in \Gam \setminus \Gamdown, \gamma_2 \leq 0, \gamma_2 \ \text{odd} \quad \text{or}\\
                                                                             \vect{\gamma} \in \Gamdown, \gamma_1 \leq m_1/2, \gamma_2 \ \text{odd}, \end{array}  \\[5mm]
                                 \cos(\gamma_{1} \theta) \sin (\gamma_{2} \vph )
                                        &  \text{if}  \quad \begin{array}{l} \vect{\gamma} \in \Gam \setminus \Gamdown, \gamma_2 > 0, \gamma_2 \ \text{even}\quad \text{or}\\
                                                                             \vect{\gamma} \in \Gamdown, \gamma_1 > m_1/2, \gamma_2 \ \text{even}, \end{array} \\[5mm]
                                 \sin(\gamma_{1} \theta) \cos (\gamma_{2} \vph )
                                        &  \text{if}  \quad \begin{array}{l} \vect{\gamma} \in \Gam \setminus \Gamdown, \gamma_2 > 0, \gamma_2 \ \text{odd}\quad \text{or}\\
                                                                             \vect{\gamma} \in \Gamdown, \gamma_1 > m_1/2, \gamma_2 \ \text{odd}. \end{array}
                                      \end{array} \right.
\end{equation*}
Evaluating the functions $\polbascontreal$ at the spherical coordinates $(\theta^{(m_1)}_{i_1}, \vph^{(m_2)}_{i_2})$ and comparing it with the definition given in \eqref{1702291124},
we obtain the identity
\begin{equation*}
\polbascontreal(\theta^{(m_1)}_{i_1}, \vph^{(m_2)}_{i_2}) = \polbasreal(\vect{i}) \qquad \text{ for $\vect{\gamma}\in\Gam$ and $\vect{i}\in \Immp$}.
\end{equation*}
Based on our experience with the complex valued basis, it makes sense to introduce the interpolation spaces as
\[\Pimreal = \mathrm{span} \left\{\, \polbascontreal \,\left|\, \vect{\gamma} \in \Gam \right. \right\}\]
and the Lagrange functions $L^{(\vect{m})}_{\mathcal{R},\vect{i}}(\theta,\vph)$ as
\begin{equation*} 
 L^{(\vect{m})}_{\mathcal{R},\vect{i}}(\theta,\vph) = \frac{1}{m_1m_2} \sum_{\vect{\gamma} \in \Gam} \frac{\polbasreal(\vect{i})}{\|\polbasreal\|_{\omega^{(\vect{m})}}^2}
 \polbascontreal(\theta,\vph)
 \quad (\theta,\vph) \in [0,\pi] \times [0,2\pi).
\end{equation*}
The corresponding reduced subspace $\Pim_{\mathcal{R},\mathrm{S}}$ and Lagrange functions $L^{(\vect{m})}_{\mathcal{R},\mathrm{S},\vect{i}}$ are
defined in the same way as in \eqref{eq:186181745} and \eqref{1508220012}, respectively.
In analogy to Theorem \ref{201512131945}, we get the following result.

\begin{theorem}
Let $\vect{m} \in \Nn^2$, $m_2$ be even, and $f\in \funpol(\Imm)$. The interpolation problem \eqref{1508220011} has a unique solution
in the space $\Pimreal$ given by the function
\begin{equation*} 
P^{(\vect{m})}_{\mathcal{R},f}(\theta,\vph) = \sum_{\vect{i} \in \Immp} f({\vect{i}}) \, L^{(\vect{m})}_{\mathcal{R},\vect{i}}(\theta,\vph).
\end{equation*}
The Lagrange functions $L^{(\vect{m})}_{\mathcal{R},\vect{i}}$, $\vect{i} \in \Immp$, form a basis of the vector space $\Pimreal$. \\
If $f\in \funpols(\Immp)$, the interpolation problem \eqref{1508220011} has a solution of the form
\begin{equation*} 
P^{(\vect{m})}_{\mathcal{R},f}(\theta,\vph) = \sum_{\vect{i} \in \Imms} f({\vect{i}}) \, L^{(\vect{m})}_{\mathcal{R},\mathrm{S},\vect{i}}(\theta,\vph).
\end{equation*}
This solution is unique in the subspace $\Pim_{\mathcal{R},\mathrm{S}} \subset \Pimreal$ spanned by the functions $L^{(\vect{m})}_{\mathcal{R},\mathrm{S},\vect{i}}$, $\vect{i} \in \Imms$.
\end{theorem}

\begin{remark}
In the discrete setting both basis systems $\polbas$ and $\polbasreal$, $\vect{\gamma} \in \Gam$  span the same space $\funpol(\Immp)$. This is different in the continuous setup.
Here, we have $\Pim = \Pimreal$ if and only if $m_1$ and $m_2$ are relatively prime. If $m_1$ and $m_2$ are not relatively prime, then the real basis functions $\polbascontreal$
for $\gamma \in \Gamdown$ are linear combinations of complex basis functions $\polbascont$ in which the indices $\gamma$ are contained in both sets $\Gamdown$ and $\Gamup$.
\end{remark}

\section{Implementation of the interpolation scheme} \label{sec:implementation}

The interpolating function $P^{(\vect{m})}_{f}$ can be computed efficiently by using fast Fourier techniques. To this end, we
expand $P^{(\vect{m})}_{f}$ in the basis $\polbascont$ as
\begin{equation} \label{20170303146} P^{(\vect{m})}_{f}(\theta,\vph) = \sum_{\vect{\gamma} \in \Gam} c_{\vect{\gamma}}(f) \polbascont(\theta,\vph).\end{equation}
In this way, once the coefficients $c_{\vect{\gamma}}(f)$ are calculated, it only remains to evaluate the sum in \eqref{20170303146}.
By Theorem \ref{201512131945} and definition $\eqref{1508220009}$ we have the following decomposition:
\[P^{(\vect{m})}_{f}(\theta,\vph) = \sum_{\vect{\gamma} \in \Gam} \frac{1}{\|\polbas\|_{\omega^{(\vect{m})}}^2} \left( \frac{1}{m_1 m_2}\sum_{\vect{i} \in \Immp}
 f({\vect{i}}) \overline{\polbas(\vect{i})} \right) \polbascont(\theta,\vph).\]
Since the functions $\polbascont$ form a basis of $\Pim$, we immediately obtain
\begin{corollary} \label{Cor:112}
For $f\in\funpol(\Immp)$, the  uniquely determined coefficients $c_{\vect{\gamma}}(f)$ in the expansion \eqref{20170303146} are given by
\[c_{\vect{\gamma}}(f)=\ts \frac{1}{\|\polbas\|_{\omega^{(\vect{m})}}^2}\,\ds \langle\;\! f,\polbas \rangle_{\omega^{(\vect{m})}}.\]
\end{corollary}

\paragraph{\bf Calculation of the coefficients $c_{\vect{\gamma}}(f)$}
Based on the formula in Corollary \ref{Cor:112}, the coefficients $c_{\vect{\gamma}}(f)$ can be computed using a two dimensional Fourier transform on the finite abelian group $\Zz / (2 m_1) \times \Zz / (2 m_2)$.
We identify this group with
\[ \Jmm = \{ \vect{i} \in \Zz^2 \ | \ 0 \leq i_1 \leq 2 m_1 -1, \ 0 \leq i_2 \leq 2 m_2 -1 \}. \]
We introduce a flip operator on $\Jmm$ by defining $\vect{i}^* = (2 m_1 - i_1 \mod 2 m_1, i_2 + m_2 \mod 2 m_2)$ for $\vect{i} \in \Jmm$.
Using this operator, we can extend a function $f$ on $\Immp$ to a function $g$ on the whole group $\Jmm$
by setting
\[ g(\vect{i}) = \frac{1}{2 m_1m_2}  \left\{ \begin{array}{rl} f(\vect{i}), \quad  & \text{if}\; \vect{i}\in\Immp,
\rule[-0.65em]{0pt}{1em}\\ f(\vect{i}^*), \quad  & \text{if}\; \vect{i}^*\in\Immp,
\rule[-0.65em]{0pt}{1em}\\
       0, \quad
  & \text{otherwise}. \end{array} \right. \]
The computation of the coefficient $c_{\vect{\gamma}}(f)$ can now be reduced to the calculation of the Fourier transform $\hat{g}$ of $g$ on $\Jmm$ by using the identity
\begin{align}c_{\vect{\gamma}}(f)&=\ts \frac{1}{\|\polbas\|_{\omega^{(\vect{m})}}^2}\,\ds \langle\;\! f,\polbas \rangle_{\omega^{(\vect{m})}}
=\ts \frac{1}{\|\polbas\|_{\omega^{(\vect{m})}}^2}\,\ds \frac{1}{m_1 m_2}\sum_{\vect{i} \in \Immp} f(\vect{i}) \overline{\polbas(\vect{i})} \label{20170304}\\
& = \ts \frac{1}{\|\polbas\|_{\omega^{(\vect{m})}}^2}\,\ds \sum_{\vect{i} \in \Jmm} g(\vect{i}) \mathrm{e}^{-\imath \gamma_{1} i_1 \pi/m_{1} } \mathrm{e}^{-\imath \gamma_{2} i_2 \pi/m_{2} }
  =  \left\{ \begin{array}{rl} 2 \hat{g}(\vect{\gamma}),  & \text{if}\, \vect{\gamma}\in\Gam, \, \gamma_1 \notin \{0,m_1\},
\rule[-0.65em]{0pt}{1em}\\ \hat{g}(\vect{\gamma}),  & \text{if}\,  \vect{\gamma}\in\Gam, \, \gamma_1 \in \{0,m_1\}. \end{array} \right. \notag
\end{align}
The computation of the discrete Fourier transform $\hat{g}(\vect{\gamma}) = \sum_{\vect{i} \in \Jmm} g(\vect{i}) \mathrm{e}^{-\imath \gamma_{1} i_1 \pi/m_{1} } \mathrm{e}^{-\imath \gamma_{2} i_2 \pi/m_{2}}$ can be executed very efficiently in $\Ord(m_1m_2 \log (m_1m_2))$ arithmetic operations using
standard algorithms for the fast Fourier transform. The values for $\|\polbas\|_{\omega^{(\vect{m})}}^2$ are taken from \eqref{1508221825}.

\begin{remark} \label{rem-1}
The invariance of the function $g$ under the flip operator, i.e., $g(\vect{i}^*) = g(\vect{i})$, implies for the Fourier domain the identity
$\hat{g}(\vect{\gamma}) = (-1)^{\gamma_2} \hat{g}(2m_1 - \gamma_1, \gamma_2)$ for all $\vect{\gamma}$ in the dual group (we identify it here also with $\Jmm$).
This glide reflection symmetry of $g$ and the flip operator are already used in the first publications studying 
double Fourier series on the sphere \cite{BoerSteinberg1975,Merilees1973}. In numerical software packages as for instance in Chebfun \cite{DriscollHaleTrefethen2014}, this
symmetry is used to obtain sparse tensor-product approximations of functions on the sphere \cite{TownsendWilberWright2016}. In \cite{TownsendWilberWright2016}, the symmetry of $g$ is called block-mirror-centrosymmetric (BMC) structure.
\end{remark}

\paragraph{\bf Calculation of the real coefficients $c_{\mathcal{R},\vect{\gamma}}(f)$}
Also for the real valued basis $\polbascontreal$ we get an expansion for the interpolating polynomial $P^{(\vect{m})}_{\mathcal{R},f}$ of the form
\begin{equation*} P^{(\vect{m})}_{\mathcal{R},f}(\theta,\vph) = \sum_{\vect{\gamma} \in \Gam} c_{\mathcal{R},\vect{\gamma}}(f) \polbascontreal(\theta,\vph),\end{equation*}
in which the expansion coefficients $c_{\mathcal{R},\vect{\gamma}}(f)$ are given by
\[c_{\mathcal{R},\vect{\gamma}}(f)=\ts \frac{1}{\|\polbasreal\|_{\omega^{(\vect{m})}}^2}\,\ds \langle\;\! f,\polbasreal \rangle_{\omega^{(\vect{m})}}.\]
The calculation of the expansion coefficients can be conducted efficiently using the formula
\begin{equation*}
c_{\mathcal{R},\vect{\gamma}}(f) = \frac{1}{\|\polbasreal\|_{\omega^{(\vect{m})}}^2} \left\{ \begin{array}{ll}
\Re \hat{g}(\vect{\gamma})
                                        & \text{if}  \quad \begin{array}{l} \vect{\gamma} \in \Gam \setminus \Gamdown, \gamma_2 \leq 0, \quad \text{or}\\
                                                                             \vect{\gamma} \in \Gamdown, \gamma_1 \leq m_1/2, \end{array}  \\[5mm]
- \Im \hat{g}(\vect{\gamma})
                                        & \text{if}  \quad \begin{array}{l} \vect{\gamma} \in \Gam \setminus \Gamdown, \gamma_2 > 0, \quad \text{or}\\
                                                                             \vect{\gamma} \in \Gamdown, \gamma_1 > m_1/2. \end{array}
\end{array} \right.
\end{equation*}
This formula can be verified as in \eqref{20170304} using the real basis \eqref{1702291124} instead of the complex functions $\polbas$.
The values $\|\polbasreal\|_{\omega^{(\vect{m})}}^2$ are explicitly known from \eqref{1508221826}.

\paragraph{\bf Calculation of averaged interpolants}
Instead of using the expansion \eqref{20170303146}, it is sometimes more convenient to implement the more symmetric expansion
\begin{equation*} {P}^{(\vect{m})}_{\mathcal{A},f}(\theta,\vph) = \sum_{\vect{\gamma} \in \Gamsup} c_{\mathcal{A},\vect{\gamma}}(f) \polbascont(\theta,\vph),\end{equation*}
in which the coefficients $c_{\mathcal{A},\vect{\gamma}}(f)$ are given by
\[ c_{\mathcal{A},\vect{\gamma}}(f) = \left\{ \begin{array}{ll} c_{\vect{\gamma}}(f)/2 & \text{if} \ \vect{\gamma} \in \Gamup \cup \Gamdown, \\
c_{\vect{\gamma}}(f) & \text{for all other $\vect{\gamma} \in \Gamsup$}.   \end{array}\right.\]
In this way, it is not necessary to make a choice between $\Gamup$ and $\Gamdown$ in order to define the interpolation space. For
$\gamma \in \Gamup$ we have $\chi^{\vect{m}}_{\vect{\gamma}} = (-1)^{\gamma_2} \chi^{\vect{m}}_{(m_1 - \gamma_1,-m_2 + \gamma_2)}$ on $\Immp$,
and therefore also $c_{\vect{\gamma}}(f) = (-1)^{\gamma_2} c_{(m_1 - \gamma_1, - m_2 + \gamma_2)}(f)$. This guarantees that
${P}^{(\vect{m})}_{\mathcal{A},f}$ is also a solution of the interpolation problem \eqref{1508220011}, although in a different space than $\Pim$. A
similar strategy is of course also possible for the real valued basis $\polbascontreal$. Averaged interpolation spaces of this type
were originally used for the Morrow-Patterson-Xu points in \cite{Harris2010,Xu1996}. A more detailed discussion of this averaging related
to multivariate interpolation on Lissajous-Chebyshev nodes can be found in \cite{DenckerErb2017a}.

\paragraph{\bf The inverse transform}
From the coefficients $c_{\vect{\gamma}}(f)$, $\vect{\gamma} \in \Gam$, the values $f \in \mathcal{L}(\Immp)$ can be recovered efficiently by a second discrete Fourier transform. We give a short description of this inverse transform. Using the interpolation condition \eqref{1508220011} and \eqref{1508201411}, we have
\begin{equation*}
 f({\vect{i}}) = P^{(\vect{m})}_f (\theta^{(m_1)}_{i_1}, \vph^{(m_2)}_{i_2}) = 
 \sum_{\vect{\gamma} \in \Gam} c_{\vect{\gamma}}(f) \polbas(\vect{i}) \quad \text{for all}\quad \vect{i} \in \Immp.
\end{equation*}
Defining the discrete function $h$ on the (dual) group $\Jmm$ as 
\[ h(\vect{\gamma}) =  \left\{ \begin{array}{rl} 2 c_{\vect{\gamma}}(f), \quad  & \text{if}\; \vect{\gamma}\in\Gam, \, \gamma_1 \notin \{0,m_1\},
\rule[-0.65em]{0pt}{1em}\\
(-1)^{\gamma_2} 2 c_{\vect{\gamma}}(f), \quad  & \text{if}\; (2m_1-\gamma_1,\gamma_2)\in\Gam, \, \gamma_1 \notin \{0,m_1\},
\rule[-0.65em]{0pt}{1em}\\ 
 c_{\vect{\gamma}}(f), \quad  & \text{if}\; \vect{\gamma}\in\Gam, \, \gamma_1 \in \{0,m_1\},
\rule[-0.65em]{0pt}{1em}\\ 
       0, \quad
  & \text{otherwise}, \end{array} \right. \]
we obtain from the relation above and the definition \eqref{A1508291531} of the basis functions $\polbas$ the following discrete Fourier sum:
\begin{equation*}
 f({\vect{i}}) =  \sum_{\vect{\gamma} \in \Jmm} h(\vect{\gamma}) \mathrm{e}^{\imath \gamma_{1} i_1 \pi/m_{1} } \mathrm{e}^{\imath \gamma_{2} i_2 \pi/m_{2} } 
 \quad \text{for all}\quad \vect{i} \in \Immp.
\end{equation*} 
In this way, the function $f \in \mathcal{L}(\Immp)$ can be recovered by applying a discrete adjoint Fourier transform to $h$ on $\Jmm$. As for 
the computation of the coefficients $c_{\vect{\gamma}}(f)$,
this adjoint transform can be executed efficiently in $\Ord(m_1m_2 \log (m_1m_2))$ arithmetic operations. 
Note that by \eqref{20170304} the function $h$ corresponds
to $\hat{g}$ on $\Gam$ and $\{\vect{\gamma} \in \Jmm \ | \ (2m_1-\gamma_1,\gamma_2) \in \Gam\}$, but in general not on the entire set $\Jmm$. 

\section{Numerical condition and convergence of the interpolation scheme} \label{sec:convergence}

We provide a mathematical description of central properties of the given interpolation scheme, as its numerical condition number, its convergence rates, and the behavior at the poles of $\Sd$. The interpolation spaces $\Pim$ and $\Pimreal$
are spanned by a double Fourier basis with a glide-reflection symmetry. Our strategy is therefore to use the theory of multivariate Fourier series to derive the pursued properties. 

We consider interpolating functions $P^{(\vect{m})}_{f}$ in which the data $f$ is given by the samples of a continuous function on the sphere. In particular, if $\ff(\theta,\vph)$ describes a continuous function on $\Sd$ in spherical coordinates, we have
\begin{equation} \label{eq:186181822}
f(\vect{i}) = \ff(\theta^{(m_1)}_{i_1}, \vph^{(m_2)}_{i_2}) \quad \text{for}\quad \vect{i} \in \Immp.
\end{equation}
Clearly, $f \in \funpols(\Immp)$ and Theorem \ref{201512131945} gives a unique interpolant $P^{(\vect{m})}_{f}$ in $\Pims \subset \Pim$. 

\paragraph{\bf Behavior at the poles of the sphere} We can describe a continuous function $\ff$ on $\Sd$ as a continuous function in spherical coordinates 
$(\theta,\vph) \in [0,\pi] \times [0,2\pi]$ by using topological identifications at the boundaries. The corresponding function space is given as 
\[ C(\Sd) = \left\{\ff \in C([0,\pi]\times[0,2\pi]) \ \left| \ \begin{array}{lll} \mathrm{(i)} & \ff(\theta,0) = \ff(\theta,2\pi), & 0 \leq \theta \leq \pi, \\
                                                                       \mathrm{(ii)} & \ff(0,\vph_1) = \ff(0,\vph_2),   & 0 \leq \vph_1,\vph_2 \leq 2 \pi, \\
                                                                       \mathrm{(iii)} &\ff(\pi,\vph_1) = \ff(\pi,\vph_2),   & 0 \leq \vph_1,\vph_2 \leq 2 \pi. \end{array}
                                                                       \right. \right\} \]
The parity-modified basis functions $\polbascont$, $\vect{\gamma} \in \Gam$, are in general not contained in $C(\Sd)$. While $\polbascont \in C([0,\pi]\times[0,2\pi])$
and the periodicity $\mathrm{(i)}$ are satisfied, the pole conditions $\mathrm{(ii)}$ and $\mathrm{(iii)}$ are only satisfied if $\gamma_2$ is odd. Also 
the interpolant $P^{(\vect{m})}_{f}$ does in general not satisfy the properties $\mathrm{(ii)}$ and $\mathrm{(iii)}$ and is therefore not necessarily continuous at 
the poles of $\Sd$. The condition $P^{(\vect{m})}_{f} \in C(\Sd)$ can be guaranteed only for particular continuous functions $\ff$. An important example is 
the space $\Pim \cap C(\Sd)$. Since $P^{(\vect{m})}_{f}$ is a projection into $\Pim$, we get for $\ff \in C(\Sd) \cap \Pim$ the identity $P^{(\vect{m})}_{f} = \ff$ 
and, thus, $P^{(\vect{m})}_{f} \in C(\Sd)$. 

In the next part we will see that the discontinuities of $P^{(\vect{m})}_{f}$ at the poles do not affect 
the global convergence of the interpolation scheme if the function $\ff$ is sufficiently smooth. This guarantees that the interpolant $P^{(\vect{m})}_{f}$ and also its derivatives will approximately satisfy the conditions $\mathrm{(ii)}$ and $\mathrm{(iii)}$ with high accuracy when the frequencies $m_1$ and $m_2$ get large. In \cite{Boyd1978}, such a property at the poles is called a natural boundary condition. Although such a natural condition is sufficient for a lot of applications, there are cases in which the described singularities at the poles result in problems. This pole problem related to the usage of the parity-modified double Fourier basis as well as possible solutions are discussed in \cite{Boyd1978,Orszag1974}.

\paragraph{\bf The Lebesgue constant}
The operator norm 
\[\Lambda^{(\vect{m})} = \sup_{\|\ff\|_{\infty} \leq 1} \|P^{(\vect{m})}_{f}\|_{\infty}, \quad \text{with} \;\| \ff \|_{\infty} = \sup_{(\theta,\vph)} |\ff(\theta,\vph)|, \]
is usually referred to as Lebesgue constant or as absolute condition number of the interpolation problem \eqref{1508220011}.
It is an upper bound for the propagation of the error in the uniform norm when constructing the interpolant $P^{(\vect{m})}_{f}$ from a continuous function $\ff$. 
\begin{theorem} \label{thm:lebesgueconstant}
The Lebesgue constant $\Lambda^{(\vect{m})}$ is bounded by
\[ \Lambda^{(\vect{m})} \leq C_{\Lambda} \ln (m_1+1) \ln (m_2+1).\] 
with a constant $C_{\Lambda}$ independent of $\vect{m}$. 
\end{theorem}

\begin{proof}
We use the representations \eqref{20170303146} and \eqref{20170304} to rewrite $P^{(\vect{m})}_{f}$
in terms of a trigonometric sum. Using the convention $\hat{g}(-\gamma_1,\gamma_2) = \hat{g}(2 m_1 - \gamma_1,\gamma_2)$ and 
the glide-reflection symmetry $\hat{g}(\vect{\gamma}) = (-1)^{\gamma_2} \hat{g}(2 m_1 - \gamma_1, \gamma_2)$ of $g$, we get
\begin{align*} P^{(\vect{m})}_{f}(\theta,\vph) &= \sum_{\vect{\gamma} \in \Gam} \frac{\hat{g}(\vect{\gamma})}{\|\polbas\|_{\omega^{(\vect{m})}}^2} \polbascont(\theta,\vph)
= \sum_{\vect{\gamma} \in \Gamstar} \hat{g}(\vect{\gamma}) \mathrm{e}^{\imath ( \gamma_{1} \theta + \gamma_2 \vph) } - \hat{g}(m_1,0) \cos(m_1 \theta),
\end{align*}
where $\Gamstar = \{ \vect{\gamma} \in \Zz^2 \ | \ (|\gamma_1|, \gamma_2) \in \Gam \ \}$ is the symmetric extension of $\Gam$ from $\Nn \times \Zz$ into
$\Zz^2$. For the operator norm $\Lambda^{(\vect{m})} = \sup_{\|\ff\|_{\infty} \leq 1} \|P^{(\vect{m})}_{f}\|_{\infty}$ we get in this way the estimates
\begin{align*} 
\Lambda^{(\vect{m})} & \leq \sup_{\|\ff\|_{\infty} \leq 1} \sup_{(\theta,\vph)} \left| 
\sum_{\vect{\gamma} \in \Gamstar} \sum_{\vect{i} \in \Jmm} g(\vect{i}) \mathrm{e}^{-\imath \gamma_{1} (i_1 \pi/m_{1} - \theta) } 
\mathrm{e}^{-\imath \gamma_{2} (i_2 \pi/m_{2}-\vph) }\right| + 1 \\
&\leq  \sup_{(\theta,\vph)} \frac{1}{2 m_1 m_2} \sum_{\vect{i} \in \Jmm} \left| 
\sum_{\vect{\gamma} \in \Gamstar}  \mathrm{e}^{-\imath \gamma_{1} (i_1 \pi/m_{1} - \theta) } 
\mathrm{e}^{-\imath \gamma_{2} (i_2 \pi/m_{2}-\vph) }\right| + 1  \\
&\leq C \int_0^{2\pi} \!\!\! \int_0^{2\pi} \left| 
\sum_{\vect{\gamma} \in \Gamstar}  \mathrm{e}^{\imath (\gamma_{1} \theta + \gamma_{2} \vph) }\right| \Dx{\theta} \Dx{\vph} + 1.
\end{align*}
The last transition from a discrete sum to a continuous double integral with a constant $C >0$ independent of $\vect{m}$ is a twofold application
of a Marcinkiewicz-Zygmund inequality, see \cite[X, Theorem 7.10]{Zygmund}. The double integral in the last line is known as Fourier-Lebesgue constant of the set $\Gamstar$. Taking apart a missing subset $\{(0,\gamma_2) \ | \ |\gamma_2| \leq m_2, \ \text{$\gamma_2$ odd}\}$, the Fourier-Lebesgue 
constants of such sets were studied in \cite{DEKL2017}. From the derivations in \cite[Section 2]{DEKL2017} (the Lebesgue constant
of the missing set is bounded by $C \ln (m_2 + 1)$), we get 
\[ \int_0^{2\pi} \!\!\! \int_0^{2\pi} \left| \sum_{\vect{\gamma} \in \Gamstar}  \mathrm{e}^{\imath (\gamma_{1} \theta + \gamma_{2} \vph) }\right| \Dx{\theta} \Dx{\vph} \leq C' \ln (m_1+1) \ln (m_2+1),\]
and, thus, the statement of the theorem. \qed
\end{proof}

\paragraph{\bf Uniform convergence of the interpolation scheme}

For a continuous function $\ff$ in spherical coordinates, we denote by $P^*$ the best approximation in the space $\Pim$ given by
$ \|\ff - P^*\|_\infty = \min_{P \in \Pim} \|\ff - P\|_\infty$. Using the fact that the interpolation operator $\ff \to P^{(\vect{m})}_{f}$ 
reproduces $P^* \in \Pim$, together with the bound in Theorem \ref{thm:lebesgueconstant}, we obtain
\begin{align*} \|\ff - P^{(\vect{m})}_{f} \|_\infty &\leq \|\ff - P^* \|_\infty + \| P^* - P^{(\vect{m})}_{f} \|_\infty \\ & \leq (\Lambda^{(\vect{m})} + 1) \|f - P^* \|_\infty  
= (C_{\Lambda}+1) \ln (m_1+1) \ln (m_2+1)  \|\ff - P^* \|_\infty.
\end{align*}
If $\ff$ is $s$ times continuously differentiable on the sphere, the best error $\|\ff - P^* \|_\infty$ can be estimated using a multivariate 
version of Jackson's inequality for trigonometric functions, as for instance described in \cite[Section 5.3]{Timan1960}.  
As a consequence, we obtain the error estimate
\begin{equation} \label{eq:1806211617} \|\ff - P^{(\vect{m})}_{f}\|_\infty \leq C_{\ff,s} \ln(m_1+1) \ln(m_2+1) \left( \frac{1}{m_1^s} + \frac{1}{m_2^s}\right),\end{equation}
with a constant $C_{\ff,s}$ that depends on $\ff$ and the smoothness $s$ but not on $\vect{m}$. 
This kind of error estimate is typical for a multitude of spectral interpolation and approximation methods and yields a fast uniform convergence of
the interpolant provided the original function $\ff$ is smooth. For multivariate polynomial
interpolation on Lissajous nodes in the hypercube $[-1,1]^{\mathsf{d}}$ similar derivations can, for instance, be found in \cite{DEKL2017,Erb2015}. 
For a tensor product spectral collocation scheme on the sphere $\Sd$, a corresponding result is provided in \cite{Ganesh1998}.

\section{Clenshaw-Curtis quadrature formula} \label{sec:cc}
Using the expansion \eqref{20170303146}, we can easily derive a Clenshaw-Curtis type interpolatory quadrature formula on the sphere
based on function evaluations on $\LSm$.
In spherical coordinates the area element on the sphere $\Sd$ is given by $\sin \theta \, \mathrm{d} \theta \mathrm{d} \vph$. Then,
the tensor product structure of the basis functions $\polbascont$ yields
\[ \frac{1}{4\pi} \int_0^{2\pi} \int_0^{\pi} \polbascont(\theta,\vph) \sin \theta \, \mathrm{d} \theta \mathrm{d} \vph = \left\{ \begin{array}{ll}
 \frac{1}{1-\gamma_1^2} & \text{if $\gamma_2 = 0$, $\gamma_1$ even,} \\ 0 & \text{otherwise}.\end{array} \right.\]
Here, we used the fact that $\int_{0}^{2\pi} e^{\imath \gamma_2 \vph} \mathrm{d} \vph = 2 \pi \delta_{\gamma_2,0}$ and that
$\int_0^{\pi} \cos( \gamma_1 \theta) \sin \theta \mathrm{d} \theta = \frac{2}{1+\gamma_1^2}$ if $\gamma_1$ is even and zero otherwise.
For the interpolating function $P^{(\vect{m})}_{f}$ with the expansion \eqref{20170303146} we therefore obtain the formula
\[ \frac{1}{4\pi} \int_0^{2\pi} \int_0^{\pi} P^{(\vect{m})}_{f}(\theta,\vph) \sin \theta \, \mathrm{d} \theta \mathrm{d} \vph = \sum_{k = 0}^{\lfloor
m_1 /2 \rfloor} \frac{c_{(2k,0)}(f)}{1-4 k^2}.  \]
The coefficients $c_{(2k,0)}(f)$ on the right hand side depend only on the values $f(\vect{i})$, $\vect{i} \in \Imms$,
which, by \eqref{eq:186181822}, are linked to function values at the Lissajous nodes $\LSm$. This formula can therefore be considered
as a Clenshaw-Curtis type quadrature rule on $\Sd$ at the nodes $\LSm$. By construction, it is
an exact quadrature rule for all functions in $\Pim$. Since $c_{(k,0)}(f) = c_{\mathcal{R},(k,0)}(f)$, the same formula
holds also true with $P^{(\vect{m})}_{\mathcal{R},f}$ as an interpolating function.

\section{Application to rotation estimation on the sphere} \label{sec:numerics}

As a final application we show how the interpolation scheme presented in this manuscript can be used to estimate the rotation
of a function on the sphere based on sample values at the nodes $\LSm$.
The algorithm to estimate the Euler angles $\vect{\beta}= (\beta_1, \beta_2, \beta_3)$ of the rotation follows the scheme presented in \cite[Section 4.1]{Ullisch2012}.
Let $\ff: \Sd \to \Rr$ denote the non-rotated function and $\ff_{\mathrm{rot}}(\vect{x}) = \ff(R_{\vect{\beta}} \vect{x})$ the rotated function on the sphere
where $R_{\vect{\beta}}$ denotes the rotation matrix determined by the three Euler angles $\vect{\beta}$. To estimate $\vect{\beta}$, we consider only the data values
$f(\vect{i}) = \ff(\vect{x}_{\vect{i}}^{(\vect{m})})$ and $f_{\mathrm{rot}}(\vect{i})=\ff(R_{\vect{\beta}} \vect{x}_{\vect{i}}^{(\vect{m})})$ measured along the spherical Lissajous curve $\vect{\ell}^{(\vect{m})}_0$.
As an interpolant for the data $f$ on the unit sphere, we use the function $P_{f}^{(\vect{m})}$ (in Cartesian coordinates). In order to obtain an estimate for the Euler angles $\vect{\beta}$ we solve the non-linear least squares problem
\begin{equation} \label{eq:leastsquares} 
\sum_{\vect{i} \in \Imms} |f_{\mathrm{rot}}(\vect{i}) - P_{f}^{(\vect{m})}(R_{\vect{\beta}} \, \vect{x}_{\vect{i}}^{(\vect{m})}  )|^2 = \mathrm{min}. \end{equation}

\begin{figure}[htb]
	\centering
	\subfigure[Original function $\ff$ given in \eqref{eq:testfunction}.
	]{\includegraphics[scale=0.4]{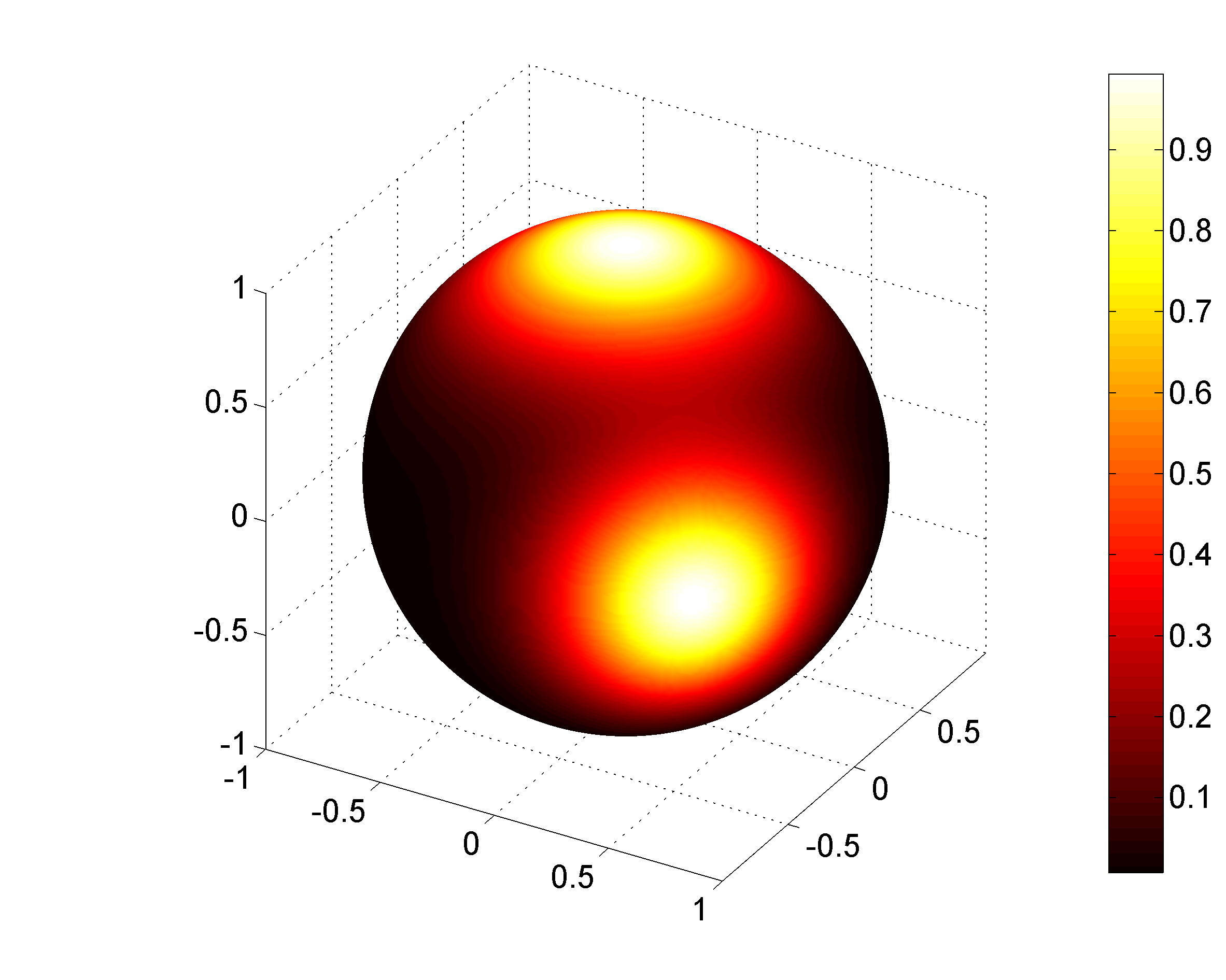}}
	\hfill	
	\subfigure[Rotated function $\ff_\mathrm{rot}$ with $\vect{\beta} = (1.4,0.2,0.9)$.
	]{\includegraphics[scale=0.4]{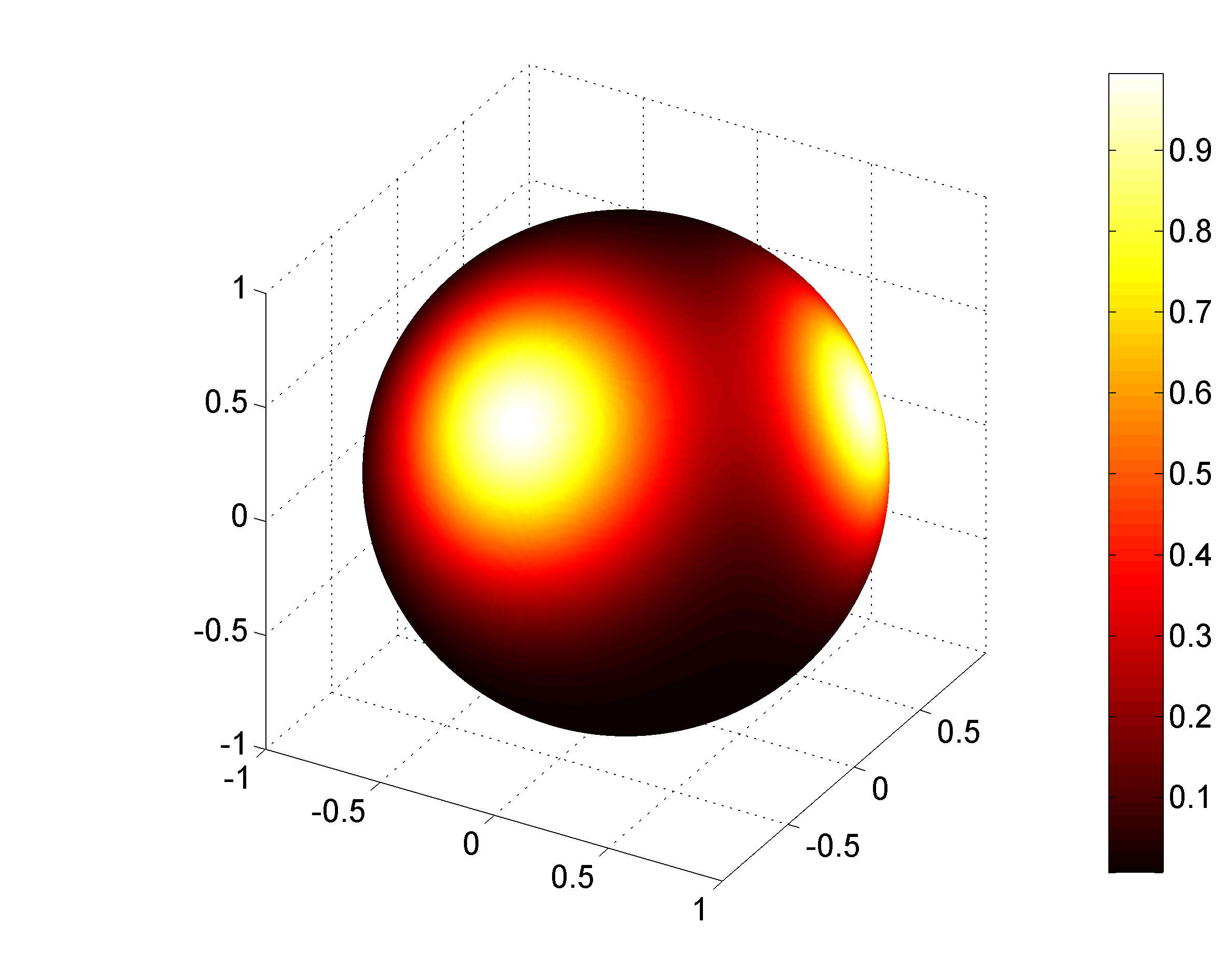}}
  	\caption{Rotation estimation on the sphere based on sample measurements on the nodes $\LSm$.
  	} \label{fig:LS-4}
\end{figure}

In the example given in Figure \ref{fig:LS-4} we used a linear combination of two Gaussians
\begin{equation} \label{eq:testfunction}
\ff(\vect{x}) = e^{-3(x^2 + y^2 + (z-1)^2)}
+e^{-4 ((x-1/\sqrt{2})^2 + (y+1/\sqrt{2})^2+z^2)}
\end{equation}
as a test function. As underlying Lissajous curve we chose $\vect{\ell}_{0}^{(15,16)}$. The non-linear least squares problem \eqref{eq:leastsquares} was
solved iteratively with a damped Gauss-Newton scheme.
With initial vector $\vect{\beta}_0 = (0,0,0)$, the solution $(1.4,0.2,0.9)$ is obtained after $16$ iterations and with the residual $2.9 \cdot 10^{-3}$. Note
that in general the functional \eqref{eq:leastsquares} has many local minima and particular care has to be given to the choice of the initial vector. A {\sc Matlab} 
code of the presented computational example and the developed interpolation scheme on spherical Lissajous nodes is provided at 
{\it https://github.com/WolfgangErb/LSphere}.

\begin{table} \centering

    \begin{tabular}{ | r | r | r || r | r | r |}
    \hline
    $\vect{m}$ & $\# \LSm$ & $ \| P_{f}^{(\vect{m})} - \ff \|_\infty$ & $\vect{m}$ & $\# \LSm$ & $ \| P_{f}^{(\vect{m})} -\ff \|_\infty$ \\
    \hline
( 3,  4) &        16 & 0.89150031122784 &  (23, 24) &       576 & 0.00000145422054 \\
( 7,  8) &        64 & 0.17505763622726 &  (27, 28) &       784 & 0.00000003014093 \\
(11, 12) &       144 & 0.01926746577677 &  (31, 32) &      1024 & 0.00000000047887 \\
(15, 16) &       256 & 0.00126029913111 &  (35, 36) &      1296 & 0.00000000000604 \\
(19, 20) &       400 & 0.00005152647682 &  (39, 40) &      1600 & 0.00000000000006 \\
\hline
    \end{tabular}
      	\caption{Approximation error $ \| P_{f}^{(\vect{m})} - \ff \|_\infty$ for the smooth function $\ff$ in \eqref{eq:testfunction}.
  	} \label{tab:LS-4}
\end{table}

\section*{Acknowledgments}
I want to thank both referees very much for their excellent work. Their suggestions helped me a lot to improve and extend this manuscript.


\begin{thebibliography}{10}

\bibitem{BoerSteinberg1975}
{\sc Boer, G., and Steinberg, L.}
\newblock {Fourier series on spheres}.
\newblock {\em Atmosphere 13}, 4 (1975), 180--191.

\bibitem{BosDeMarchiVianelloXu2006}
{\sc Bos, L., Caliari, M., De~Marchi, S., Vianello, M., and Xu, Y.}
\newblock {Bivariate Lagrange interpolation at the Padua points: the generating
  curve approach.}
\newblock {\em J. Approx. Theory 143}, 1 (2006), 15--25.

\bibitem{Boyd1978}
{\sc Boyd, J.~P.}
\newblock {The choice of spectral functions on a sphere for boundary and
  eigenvalue problem: A comparison of Chebyshev, Fourier and associated
  Legendre expansions}.
\newblock {\em Mon. Wea. Rev. 106\/} (1978), 1184--1191.

\bibitem{Boyd2000}
{\sc Boyd, J.~P.}
\newblock {\em {Chebyshev and Fourier spectral methods.}}
\newblock Dover Publications Inc., New York, 2000.

\bibitem{CaliariDeMarchiVianello2005}
{\sc Caliari, M., De~Marchi, S., and Vianello, M.}
\newblock {Bivariate polynomial interpolation on the square at new nodal sets.}
\newblock {\em Appl. Math. Comput. 165}, 2 (2005), 261--274.

\bibitem{Costa2010}
{\sc Costa, A.~F., Yen, Y.-F., and Drangova, M.}
\newblock {Registering spherical navigators with spherical harmonic expansions
  to measure three-dimensional rotations in Magnetic Resonance Imaging}.
\newblock {\em Magnetic Resonance Imaging 28}, 2 (2010), 185--194.

\bibitem{DenckerErb2017a}
{\sc Dencker, P., and Erb, W.}
\newblock {A unifying theory for multivariate polynomial interpolation on
  general Lissajous-Chebyshev nodes}.
\newblock {\em arXiv:1711.00557 [math.NA]\/} (2017).

\bibitem{DenckerErb2015a}
{\sc Dencker, P., and Erb, W.}
\newblock {Multivariate polynomial interpolation on Lissajous-Chebyshev nodes}.
\newblock {\em J. Appr. Theory 219\/} (2017), 15--45.

\bibitem{DEKL2017}
{\sc Dencker, P., Erb, W., Kolomoitsev, Y. and Lomako, T.}
\newblock { Lebesgue constants for polyhedral sets and polynomial interpolation on Lissajous-Chebyshev nodes}.
\newblock {\em Journal of Complexity 43\/} (2017), 1--27.

\bibitem{DriscollHaleTrefethen2014}
{\sc Driscoll, T.~A., Hale, N., and Trefethen, L.~N.} (editors) 
\newblock {\em Chebfun Guide}.
\newblock Pafnuty Publications, Oxford, 2014.

\bibitem{Erb2015}
{\sc Erb, W.}
\newblock {Bivariate Lagrange interpolation at the node points of Lissajous
  curves - the degenerate case}.
\newblock {\em Appl. Math. Comput. 289\/} (2016), 409--425.

\bibitem{ErbKaethnerAhlborgBuzug2015}
{\sc Erb, W., Kaethner, C., Ahlborg, M., and Buzug, T.~M.}
\newblock {Bivariate Lagrange interpolation at the node points of
  non-degenerate Lissajous curves}.
\newblock {\em Numer. Math. 133}, 1 (2016), 685--705.

\bibitem{Fornberg1995}
{\sc Fornberg, B.}
\newblock {A pseudospectral approach for polar and spherical geometries}.
\newblock {\em SIAM J. Sci. Comp. 16\/} (1995), 1071--1081.

\bibitem{Ganesh1998}
{\sc Ganesh, M., Graham, I. and Sivaloganathan J.}
\newblock {A new spectral boundary integral collocation method for three-dimensional potential problems}.
\newblock {\em SIAM J. Numerical Analysis 35} (1998), 778--804.

\bibitem{Harris2010}
{\sc Harris, L.~A.}
\newblock {Bivariate Lagrange interpolation at the Chebyshev nodes.}
\newblock {\em Proc. Am. Math. Soc. 138}, 12 (2010), 4447--4453.

\bibitem{Merilees1973}
{\sc Merilees, P.~E.}
\newblock {The pseudospectral approximation applied to the shallow water wave
  equations on a sphere}.
\newblock {\em Atmosphere 11\/} (1973), 13--20.

\bibitem{MorrowPatterson1978}
{\sc Morrow, C.~R., and Patterson, T.~N.~L.}
\newblock {Construction of algebraic cubature rules using polynomial ideal
  theory.}
\newblock {\em SIAM J. Numer. Anal. 15\/} (1978), 953--976.

\bibitem{Orszag1974}
{\sc Orszag, S.~A.}
\newblock {Fourier series on spheres}.
\newblock {\em Monthly Weather Review 102\/} (1974), 56--75.

\bibitem{Timan1960}
A.~F.~Timan, {Theory of approximation of functions of a real variable}, translated by J. Berry, Pergamon Press, Oxford, 1963.

\bibitem{TownsendWilberWright2016}
{\sc Townsend, A., Wilber, H., and Wright, G.}
\newblock Computing with functions in spherical and polar coordinates {I}.
  {T}he sphere.
\newblock {\em SIAM J. Sci. Comp. 38}, 4 (2017), C403--C425.

\bibitem{Trefethen2017}
{\sc Trefethen, L. N.}
\newblock Multivariate polynomial approximation in the hypercube.
\newblock {\em Proc. Amer. Math. Soc. 145} (2017),  4837--4844.

\bibitem{Ullisch2012}
{\sc Ullisch, M.}
\newblock {\em {A navigator based rigid body motion correction for magnetic
  resonance imaging}}.
\newblock Dissertation, {Technische Hochschule Aachen}, 2012.

\bibitem{Welchetal2002}
{\sc Welch, E.~B., Manduca, A., Grimm, R.~C., Ward, H.~A., and Clifford,
  R.~J.~J.}
\newblock {Spherical navigator echoes for full 3D rigid body motion measurement
  in MRI}.
\newblock {\em Magnetic Resonance in Medicine 47\/} (2002), 32--41.

\bibitem{Xu1996}
{\sc Xu, Y.}
\newblock {Lagrange interpolation on Chebyshev points of two variables.}
\newblock {\em J. Approx. Theory 87}, 2 (1996), 220--238.

\bibitem{Zygmund}
\sc{Zygmund, A.} 
\newblock {\em {Trigonometric series, third edition, Volume I \& II combined}.}
\newblock Cambridge University Press, Cambridge, 2002.

\end{thebibliography}
\end{document}